\documentclass[a4paper,12pt]{article}
    \usepackage{amsmath}
    \usepackage{amssymb}
    \usepackage{amsthm}

\usepackage{graphicx,psfrag}
\usepackage{pstricks}
\usepackage{subfigure}

      \newcommand {\al}   {\alpha}          
                
      \newcommand {\del}  {\delta}          
              \newcommand {\ve}   {\varepsilon}

                \newcommand {\Sig}    {\Sigma}
      \newcommand {\om}   {\omega}          \newcommand {\Om}  {\Omega}
      \newcommand {\pl}   {\partial}        
           \newcommand {\UUU}  {U}
      \newcommand {\RRR}  {{\mathbb R}}     \newcommand {\AAA}  {A}              
                     \newcommand{\Convex}{\mathcal{C}}
      \newcommand {\ZZZ}  {{\mathbb Z}}     \newcommand {\DDD}  {{D}}            \newcommand{\NNN}{\mathcal{N}}
              \newcommand {\BBB}  {B}              \newcommand {\PPP}{\hat{P}}
      \newcommand {\CCC}  {C}       \newcommand {\SSS}  {{S}}
              
              \newcommand {\vk}  {\varkappa}
      \newcommand {\trap}{\text{\pspolygon[linewidth=0.4pt](-0.02,0)(0.32,0)(0.23,0.22)(0.07,0.22)\ \ \,}}
      \newcommand {\trapp}{\text{\scalebox{0.6}{\pspolygon[linewidth=0.4pt](-0.02,0)(0.32,0)(0.23,0.22)(0.07,0.22)}\ \ \,}}
      \newcommand {\omom}   {k}

      \newcommand {\bbss}  {\begin{slide}}
      \newcommand {\eess}  {\end{slide}}
      \newcommand {\beq}  {\begin{equation}}
      \newcommand {\eeq}  {\end{equation}}
      \newcommand {\beqo}  {\begin{equation*}}
      \newcommand {\eeqo}  {\end{equation*}}

      \newtheorem{theorem}{Theorem}
      \newtheorem{lemma}{Lemma}
      
      \newtheorem{definition}{Definition}
       \newtheorem{remark}{Remark}
      \newtheorem{problem}{Problem}
      \newtheorem{propo}{Proposition}

\author{Alexander Plakhov\thanks{Center for R\&D in Mathematics and Applications, Department of Mathematics, University of Aveiro, Portugal and Institute for Information Transmission Problems, Moscow, Russia}}

\title{Newton's problem of minimal resistance under the single-impact assumption}

\date{}

\begin{document}

\maketitle

\begin{abstract}
A parallel flow of non-interacting point particles is incident on a body at rest. When hitting the body's surface, the particles are reflected elastically. Assume that each particle hits the body at most once (SIC condition); then the force of resistance of the body along the flow direction can be written down in a simple analytical form.

The problem of minimal resistance within this model was first considered by Newton \cite{N} in the class of bodies with a fixed length $M$ along the flow direction and with a fixed maximum orthogonal cross section $\Om$, under the additional conditions that the body is convex and rotationally symmetric. Here we solve the problem (first stated in \cite{BFK}) for the wider class of bodies satisfying SIC and with the additional conditions removed. The scheme of solution is inspired by Besicovitch's method of solving the Kakeya problem \cite{Bes}. If $\Om$ is a disc, the decrease of resistance as compared with the original Newton problem is more than twofold; the ratio tends to 2 as $M \to 0$ and to $20.25$ as $M \to \infty.$ We also prove that the infimum of resistance is 0 for a wider class of bodies with both single and double impacts allowed.
\end{abstract}

\begin{quote}
{\small {\bf Mathematics subject classifications:} 49Q10, 49K30}
\end{quote}

\begin{quote}
{\small {\bf Key words and phrases:} Newton's problem of minimal resistance, shape optimization, Kakeya problem, billiards.}
\end{quote}

\section{Introduction}

{\bf 1.1.} Consider a bounded domain (a {\it body}) $B$ in Euclidean space $\RRR^3$ and a parallel flow of point particles with unit velocity incident on $B$. If a particle hits the body at a regular point of the boundary $\pl B$, it is reflected according to the elastic (billiard) law. A particle can make several reflections from the body. The particles do not interact with each other.

Under some additional assumptions (if, for example, the body $B$ is convex) and knowing the flow density, it is possible to determine the force of pressure of the flow on the body. This force is usually called the {\it force of resistance}. One is traditionally interested in finding the body, in a prescribed class of bodies, that minimizes the projection of this force on the flow direction. This projection is also called the {\it resistance}.

Remarkably, this simple mechanical model is a source of various problems from different areas of mathematics. First stated by Newton \cite{N} in a class of convex axisymmetric bodies, the problem of minimal resistance became one of the problems that gave origin to the calculus of variations. With the symmetry assumption removed, for various classes of convex bodies one comes to unusual and interesting multidimensional variational problems. They have been intensively studied in 1990s and 2000s (see \cite{BrFK}-\cite{LP1}). 

The condition of convexity guarantees absence of multiple collisions and allows one to write down the problem in a convenient analytical form. In the case of nonconvex bodies multiple reflections may occur, and one often needs to use methods of the billiard theory \cite{JDCS,CJM,ebook}. If, additionally, it is allowed to vary the direction $v$ of the flow and one is interested in minimizing the resistance averaged over $v$, one comes to interesting problems related to optimal mass transfer \cite{ARMA,1dimMK,ebook}.

Here we are going to study several problems of minimal resistance for bodies that are (generally) non-symmetric and nonconvex, but satisfy the so-called {\it single impact condition} (SIC): a particle cannot make more than one reflection from the body. This condition assures that the standard analytic formula for the resistance is preserved. Convex bodies obviously satisfy this condition, but not only they: if, for instance, a normal vector at each regular point of the part of $\pl B$ faced to the flow makes an angle smaller than $\pi/6$ with the flow direction, then $B$ satisfies SIC.

In brief, here we consider several classes of bodies with a fixed length along the direction of the flow and fixed maximum orthogonal cross section. We find the infimum of resistance under the additional assumptions that the body is the subgraph of a function and satisfies SIC, and show that the infimum remains unchanged when SIC is relaxed. Finally, we prove that the infimum is zero in the wider class of bodies (not necessarily subgraphs) when SIC is replaced with the so-called {\it double impact condition}.
\vspace{2mm}

{\bf 1.2.} More precisely, consider an orthonormal reference system $x_1,\, x_2,\, z$ in $\RRR^3$ and denote $x = (x_1,x_2)$. We assume that the flow falls vertically downward with velocity $v = (0,0,-1)$.

\begin{definition}\label{o_SIC}\rm
Let $\Om$ be a convex open bounded set in $\RRR^2$. We say that a piecewise smooth function $u : \bar\Om \to \RRR$ satisfies the {\it single impact condition} (SIC), if for any regular point $x \in \Om$ and any $t > 0$ such that $x - t\nabla u(x) \in \bar\Om$,
\beq\label{o SIC}
\frac{u(x - t\nabla u(x)) - u(x)}{t} \le \frac{1}{2} (1 - |\nabla u(x)|^2).
\eeq
\end{definition}

\begin{remark}\label{zam 01}\rm
According to this definition, the trajectory of a particle after a reflection from the graph of $u$ is situated above the graph. It may, however, touch graph$(u)$ (this happens when the inequality in (\ref{o SIC}) turns into equality); this is not considered to be a reflection.
\end{remark}

\begin{definition}\label{o_S}\rm
Let $M > 0$. We denote by $\SSS_{\Om,M}$ the class of functions $u : \bar{\Om} \to \RRR$ satisfying SIC and such that $0 \le u(x) \le M$ for all $x \in \bar{\Om}$.
\end{definition}

The resistance of a function $u \in \SSS_{\Om,M}$ is defined by
\beq\label{o Resist}
F(u) = \int_\Om \frac{dx}{1 + |\nabla u(x)|^2}.
\eeq

\begin{remark}\label{zam 1}\rm
One can think of a 3D body bounded above by the graph of $u$. There are at least two reasonable ways of defining such a body: either
\beq\label{B_u}
B_u = \{ (x,z) :\, x \in \bar{\Om}, \ 0 \le z \le u(x) \} \subset \RRR^3,
\eeq
or
\beq\label{barb_u}
\bar B_u = \{ (x,z) :\, x \in \bar{\Om}, \ z \le u(x) \}.
\eeq
If $u \in \SSS_{\Om,M}$ (and therefore $u$ satisfies SIC), then in both cases (\ref{B_u}) and (\ref{barb_u}) the vertical component of the momentum imparted to the body by a flow particle is proportional to the integrand in (\ref{o Resist}) (with the mass of the particle being the ratio). Summing up the momenta imparted by all incident particles per unit time, one concludes that the vertical component of the body's resistance force equals $2\rho F(u)$, where $\rho$ is the flow density. That is, formula (\ref{o Resist}) is in agreement with the physical meaning of resistance. On the contrary, if SIC is not satisfied, formula (\ref{o Resist}) has no physical meaning.
\end{remark}

The following problem naturally appears.

\begin{problem}\label{p S}
Find $\inf_{u\in\SSS_{\Om,M}} F(u)$.
\end{problem}

The condition SIC and Problem \ref{p S} were first stated in \cite{BFK} and further discussed in the papers \cite{CL1}-\cite{Kawohl}.

Our aim in Problem \ref{p S} is to minimize the resistance over a class of bodies with the fixed maximum cross section $\Om$ orthogonal to the direction $v$ of the flow and fixed length $M$ along this direction. If $M$ goes to infinity, the infimum goes to zero; to see it, it suffices to consider a sequence of cones with the base $\Om$ and height $M \to \infty$.

\begin{remark}\label{zam 3}\rm
It is possible to consider {\it nonconvex} domains $\Om$, define the condition SIC, and formulate the minimization problem for the class of corresponding functions. However, this consideration would lead to technical complications. For instance, SIC would have different forms depending on the physical interpretation ($B_u$ or $\bar B_u$) of the body. If $\Om$ is convex, SIC does not depend on this interpretation.
\end{remark}

Take the body $B_u$ (\ref{B_u}) corresponding to a piecewise smooth function $u : \bar{\Om} \to \RRR$ and consider the billiard in $\RRR^3 \setminus B_u$. Let a billiard particle initially move according to $x(t) = x$,\, $z(t) = -t$, then make several reflections (maybe none) at regular points of $\pl B_u$, and finally move freely; we denote its final velocity by $v^+(x;u) = (v_1^+(x;u),\, v_2^+(x;u),\, v_3^+(x;u))$. One obviously has $v^+(x;u) = v$ for $x \not\in \bar{\Om}$ (recall that $v = (0,0,-1)$).

\begin{definition}\label{o_scatt_reg}\rm
We say that the billiard scattering is {\it regular}, if the function $v^+(\cdot;u)$ is defined on a full-measure subset of $\RRR^2$ and is measurable.
\end{definition}

\begin{definition}\label{o_U}\rm
We denote by $\UUU_{\Om,M}$ the class of piecewise smooth functions $u : \bar{\Om} \to \RRR$ such that

(a) $0 \le u(x) \le M$ for all $x \in \bar{\Om}$ and

(b) the corresponding billiard scattering is regular.
\end{definition}

\begin{remark}\label{zam 4}\rm
It is easy to provide a function $u$ that does not satisfy (b). Suppose that a part of {\rm graph}$(u)$ is a piece of a paraboloid of rotation whose focus coincides with a singular point of {\rm graph}$(u)$. Then the function $v^+(\cdot;u)$ is not defined in the projection of that piece of paraboloid on the $x$-plane.
\end{remark}

The resistance of $u \in \UUU_{\Om,M}$ is defined by
\beq\label{o Res2}
F(u) = \int_\Om \frac{1 + v_3^+(x;u)}{2}\, dx.
\eeq
This formula has a strong physical meaning. Indeed, a particle with mass $\mu$ that initially moves according to $x(t) = x$,\, $z(t) = -t$, imparts the momentum $\mu v - \mu v^+(x;u)$ to the body, and the projection of this momentum on the direction of $v$ equals
$$
\mu \langle v - v^+(x;u), \ v \rangle = \mu (1 + v_3^+(x;u));
$$
here and in what follows $\langle \cdot\,, \cdot \rangle$ means scalar product. Summing all the imparted momenta, one finds that the third component of the physical force of resistance of the body $B_u$ to a flow with constant density $\rho$ equals $-2\rho F(u)$.

One has $\SSS_{\Om,M} \subset \UUU_{\Om,M}$, and for $u \in \SSS_{\Om,M}$ formulae (\ref{o Resist}) and (\ref{o Res2}) give the same value. Indeed, if $u \in \SSS_{\Om,M}$ then $v^+(x;u)$ is defined for all regular points $x$ of $u$ and, moreover, can be determined explicitly. Namely, a normal vector to graph$(u)$ at $(x,u(x))$ is $n = (\nabla u(x), -1)$, and
$$
v^+(x;u) = v - \frac{2\langle v, \, n\rangle}{|n|^2}\, n = \frac{1}{1 + |\nabla u(x)|^2} \ \big(-2\nabla u(x), \ 1 - |\nabla u(x)|^2\big).
$$
Therefore $(1 + v_3^+(x;u))/2 = 1/(1 + |\nabla u(x)|^2)$, and so, the integrals in the right hand sides of (\ref{o Resist}) and (\ref{o Res2}) coincide. This means that the value $F(u)$ is well defined.

We have the following problem.

\begin{problem}\label{p U}
Find $\inf_{u\in\UUU_{\Om,M}} F(u)$.
\end{problem}

From the mechanical point of view, it makes sense to consider a wider class of bodies whose surface faced to the flow is not necessarily the graph of a function. For a 3D body $B$ we again consider the billiard in $\RRR^3 \setminus B$ and analogously define the notion of regular scattering in terms of the final velocity $v^+(x;B) = (v_1^+(x;B),\, v_2^+(x;B),\, v_3^+(x;B)), \ x \in \RRR^2$.

\begin{definition}\label{o_B}\rm
Denote by $\BBB_{\Om,M}$ the class of 3D domains $B$ (bodies) such that the surface $\pl B \setminus (\pl\Om \times \RRR)$ is piecewise smooth and

(a) $\bar{\Om} \times \{0\} \subset B \subset \bar{\Om} \times [0,\, M]$;

(b) the corresponding billiard scattering is regular.
\end{definition}

\begin{remark}\label{zam 5}\rm
The convex set $\Om$ is not necessarily piecewise smooth. It may even happen that the set of singular points of $\pl\Om$ is everywhere dense in $\pl\Om$. In that case the part of the boundary $\pl B$ that belongs to the cylinder $\pl\Om \times \RRR$ contains an everywhere dense set of singular points. That is why we only require that the complementary part of the boundary $\pl B$ is piecewise smooth.
\end{remark}

The resistance of a body $B \in \BBB_{\Om,M}$ is defined by
\beq\label{o Res3}
R(B) = \int_\Om \frac{1 + v_3^+(x;B)}{2}\, dx.
\eeq

\begin{definition}\label{o_u_B}\rm
For a body $B \in \BBB_{\Om,M}$ we define the function $u_B : \bar{\Om} \to \RRR$ by
$$
u_B(x) = \sup \{ z : (x,z) \in B \}
$$
(see Fig.~\ref{fig_uB}).
\end{definition}

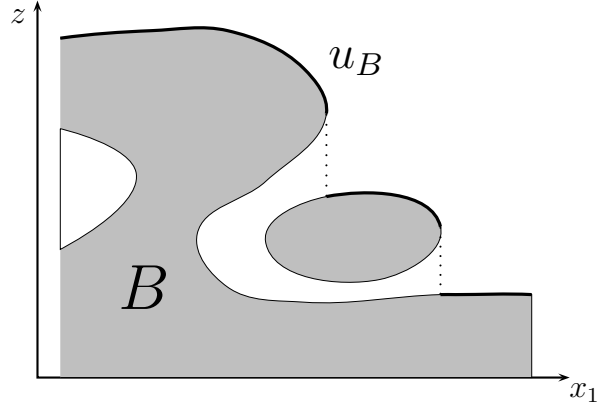
\begin{figure}[h]
\begin{picture}(0,145)
\rput(4,0.2){
\scalebox{1}{
\psline{<->}(-0.5,5)(-0.5,0)(6.5,0)
\pscustom[linewidth=0pt,fillstyle=solid,fillcolor=lightgray]{
\psecurve[linewidth=0pt](-0.7,4.4)(-0.2,4.5)(1,4.6)(2,4.6)(3,4.1)(3.3,3.5)(2.5,2.6)(1.6,1.9)(2,1.2)(3,1)(4.8,1.1)(5,1.1)(6,1.1)(7,1)
\psline(6,0)(-0.2,0)
}
\pscustom[linewidth=0pt,fillstyle=solid,fillcolor=white]
{\pscurve[linewidth=0pt](-0.2,3.3)(0.8,2.7)(-0.2,1.7)
\psline[linewidth=0pt](-0.2,1.7)(-0.2,3.3)}
\psecurve[linewidth=0pt,fillstyle=solid,fillcolor=lightgray](4,1.3)(4.8,2)(3.3,2.4)(2.5,1.8)(4,1.3)(4.8,2)(3.3,2.4)
\psecurve[linewidth=1.2pt](-0.7,4.4)(-0.2,4.5)(1,4.6)(2,4.6)(3,4.1)(3.3,3.5)(2.5,2.6)
\psline[linestyle=dotted](3.3,3.5)(3.3,2.4)
\psline[linestyle=dotted](4.8,2)(4.8,1.1)
\psecurve[linewidth=1.2pt](4,1.3)(4.8,2)(3.3,2.4)(2.5,1.8)
\psecurve[linewidth=1.2pt](3,1)(4.8,1.1)(5,1.1)(6,1.1)(7,1)
\rput(0.9,1.2){\scalebox{2}{$B$}}
\rput(3.7,4.2){\scalebox{1.5}{$u_B$}}
\rput(6.7,-0.2){\scalebox{1}{$x_1$}}
\rput(-0.75,4.8){\scalebox{1}{$z$}}
}
}
\end{picture}
\caption{The cross section of a body $B$ and of graph$(u_B)$ by a vertical plane $x_2 =$\,const.}
\label{fig_uB}
\end{figure}

The point of first reflection of a flow particle from the body always lies on {\rm graph}$(u_B)$. The condition SIC for $u_B$ guarantees that the trajectory of the particle after the first reflection lies above graph$(u_B)$ and therefore there are no further reflections from $B$.

\begin{remark}\label{zam 6}\rm
The class of bodies $\BBB_{\Om,M}$ is, in a sense, "larger" than the class of functions $\UUU_{\Om,M}$. In particular, the composition of mappings $u \mapsto B_u \mapsto u_{B_u}$ is the identity of $\UUU_{\Om,M}$, however the composition $B \mapsto u_B \mapsto B_{u_B}$ {\rm is not} the identity of $\BBB_{\Om,M}$, but rather a projection onto a proper subset of $\BBB_{\Om,M}$.
\end{remark}

\begin{definition}\label{o_B}\rm
We denote by $\Sig_{\Om,M}$ the class of bodies $B \in \BBB_{\Om,M}$ such that $u_B$ satisfies SIC.
\end{definition}

The following proposition states that the infima of resistance over {\it bodies} satisfying SIC and over {\it functions} satisfying SIC coincide.

\begin{propo}\label{propo}
$\inf_{B\in\Sig_{\Om,M}} R(B) = \inf_{u\in\SSS_{\Om,M}} F(u)$.
\end{propo}

\begin{proof}
The proof is quite easy. It suffices to note that

(i) if $B \in \Sig_{\Om,M}$ then $u_B \in \SSS_{\Om,M}$ and $R(B) = F(u_B)$, and

(ii) if $u \in \SSS_{\Om,M}$ then $B_u \in \Sig_{\Om,M}$ and $F(u) = R(B_u).$
\end{proof}

\begin{problem}\label{p B}
Find $\inf_{B\in\BBB_{\Om,M}} R(B)$.
\end{problem}

Actually, a problem very similar to Problem \ref{p B} was solved in \cite{ebook}, and the method used there can be easily adapted to show that $\inf_{B\in\BBB_{\Om,M}} R(B) = 0$. In other words, if one allows multiple reflections, the resistance can be made arbitrarily small.
It is then natural to fix the maximal allowed number of reflections and study the corresponding problem. Surprisingly enough, even if only single and double reflections are allowed, the infimum of resistance equals zero.

\begin{definition}\label{o_D}\rm
We say that a body $B \in \BBB_{\Om,M}$ satisfies the {\it double impact condition} (DIC), if each incident particle with the initial velocity $v = (0,0,-1)$ has no more than two reflections from $B$. The class of connected bodies satisfying DIC is denoted by $\DDD_{\Om,M}$.
\end{definition}

\begin{problem}\label{p D}
Find $\inf_{B\in\DDD_{\Om,M}} R(B)$.
\end{problem}

It is not difficult to find a lower bound for the resistance of a function $u \in \SSS_{\Om,M}$. To that end, put $d(x) = \text{dist}(x, \pl\Om)$ for $x \in \Om$ and define
\beq\label{o_phi}
\phi(\Om,M) = \int_\Om \frac 12 \bigg( 1 - \frac{M}{\sqrt{M^2 + d^2(x)}}\bigg)\, dx.
\eeq
It was first noticed in \cite {BFK} that
\beq\label{ineq1}
\inf_{u \in \SSS_{\Om,M}} F(u) \ge \phi(\Om,M).
\eeq
For the sake of the reader's convenience, here we reproduce the proof of (\ref{ineq1}).

Let a particle be reflected from a regular point $(x,z)$ of the graph of $u \in \SSS_{\Om,M}$. It may further happen that the particle does or does not intersect the horizontal plane $z = 0$. In the former case let $(x',0)$ be the point of intersection; then the final velocity is
$$
v^+(x;u) = \frac{1}{\sqrt{z^2 + |x'-x|^2}}\, (x' - x,\, -z).
$$
We have $x' \not\in \Om$  and therefore $|x-x'| \ge d(x)$, and $0 \le z \le M$; therefore
$$
v_3^+(x;u) \ge - \frac{M}{\sqrt{M^2 + d^2(x)}}.
$$
In the latter case we have $v_3^+(x;u) \ge 0$. In both cases the integrand in (\ref{o Res2}) is greater than or equal to $\frac 12\, (1 - M/\sqrt{M^2 + d^2(x)})$, and so, $F(u) \ge \phi(\Om,M).$

The following theorems provide solutions for problems \ref{p S} -- \ref{p D}.

\begin{theorem}\label{t1}
\beq\label{ft1}
\inf_{u \in \SSS_{\Om,M}} F(u) = \phi(\Om,M).
\eeq
\end{theorem}

\begin{theorem}\label{t2}
\beq\label{ft2}
\inf_{u \in \UUU_{\Om,M}} F(u) = \phi(\Om,M) = \inf_{B \in \Sig_{\Om,M}} R(B).
\eeq
\end{theorem}

\begin{theorem}\label{t3}
\beq\label{ft3}
\inf_{B \in \BBB_{\Om,M}} R(B) = 0 =\inf_{B \in \DDD_{\Om,M}} R(B).
\eeq
\end{theorem}

These results are counterintuitive. Indeed, the statement of Theorem \ref{t1} implies that the graph of a nearly optimal function looks like a plateau with height $M$. The plateau surface is complicated and greatly inclined, with the angle of inclination being typically greater than $45^0$. The plateau is crossed with a huge number of narrow deep valleys, but the total area covered by the valleys is small. The reflected particles further move along the valleys, and their density in the valleys is very high.

Theorem \ref{t1} provides the answer for the class of bodies such that (a) {\it the front part of the body surface is the graph of a function} and (b) {\it the single impact condition is satisfied}. If only one of the conditions (a) and (b) is removed (and thus a larger class of bodies is considered), the infimum remains the same, as indicates Theorem \ref{t2}. However, if both the conditions are removed, the infimum becomes zero. It remains zero even if only single and double reflections are allowed. This is the claim of Theorem \ref{t3}.

Note that a similar problem concerning minimization of specific resistance in hollows was considered in \cite{hollows}.
\vspace{2mm}

{\bf 1.3.}
It is interesting to compare the minimizers in the four main classes of functions studied so far. We take the unit disc $\Om = \Om_0$ and consider the functions $u : \bar{\Om}_0 \to \RRR, \ 0 \le u \le M$ satisfying SIC, with the following additional conditions imposed:

(P$_{SC}$) $u$ is radially symmetric and concave (the case considered by Newton \cite{N});

(P$_{C}$) $u$ is concave \cite{BrFK,BFK,BK,LO};

(P$_{S}$) $u$ is radially symmetric \cite{BK,CL1};

(P) no additional conditions on $u$ (the present paper).
\vspace{1mm}

The minimizer exists in the classes (P$_{SC}$), (P$_{C}$), (P$_{S}$) and does not exist in the class (P).

The optimal shapes in the classes (P$_{SC}$), (P$_{C}$), and (P$_{S}$) with $M = 1$ are depicted in Figs.~\ref{fig4min}\,(a)--(c). Figure~\ref{fig4min}\,(c) is borrowed from \cite{W}. Several trajectories of flow particles are also shown there.

Nearly optimal shapes in the class (P) are extremely complicated and not easy to depict. In Fig.~\ref{fig4min}\,(d) a very schematic representation of a central vertical cross section of such a shape is given, also with $M = 1$. The particles shown in the figure after the reflection leave the plane of cross section and then move along narrow valleys (which are not shown); therefore their trajectories after the reflection are shown dashed.

\begin{figure}[h]
\centering
\includegraphics[scale=0.3]{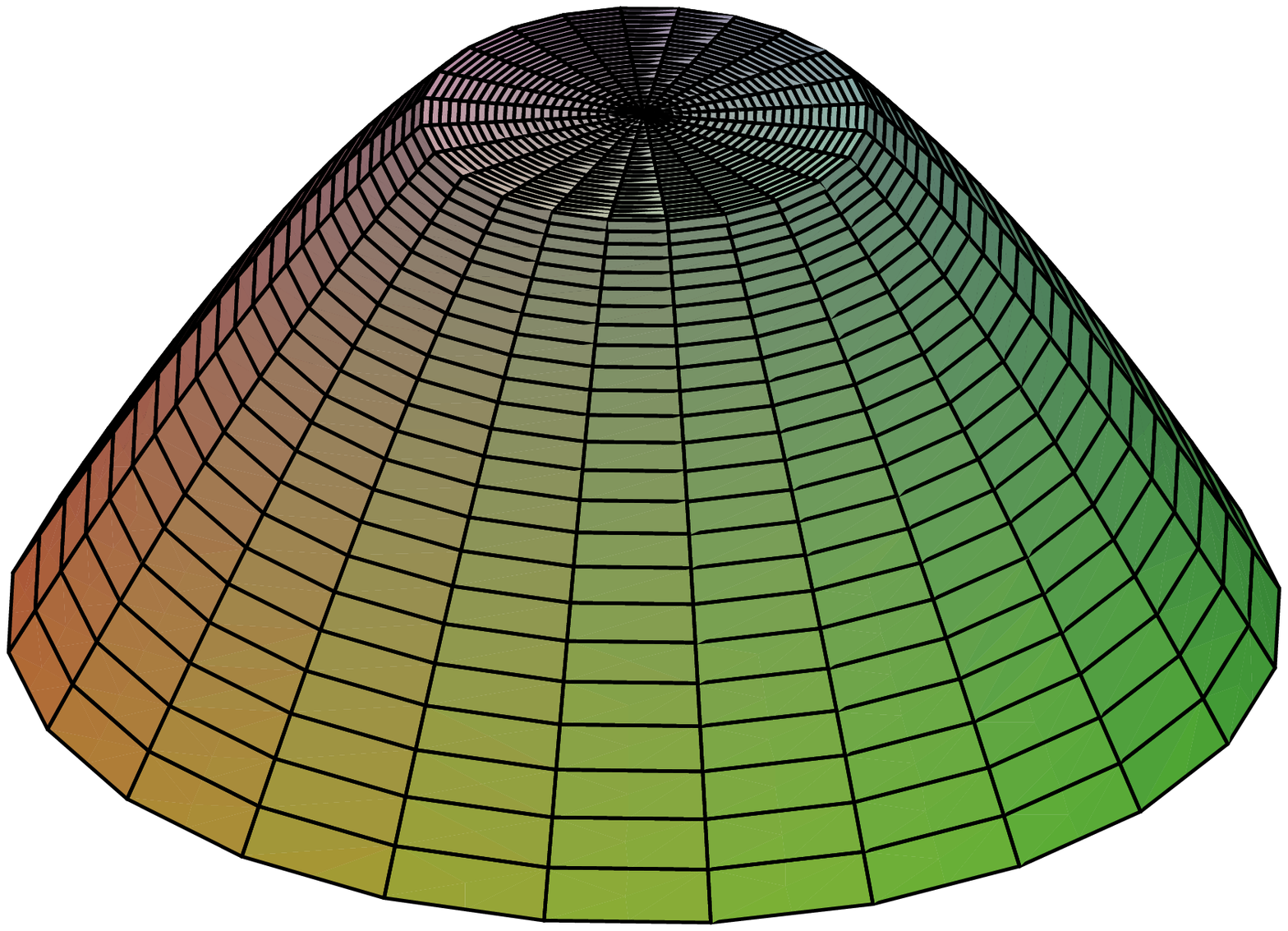}\qquad
\includegraphics[scale=0.31]{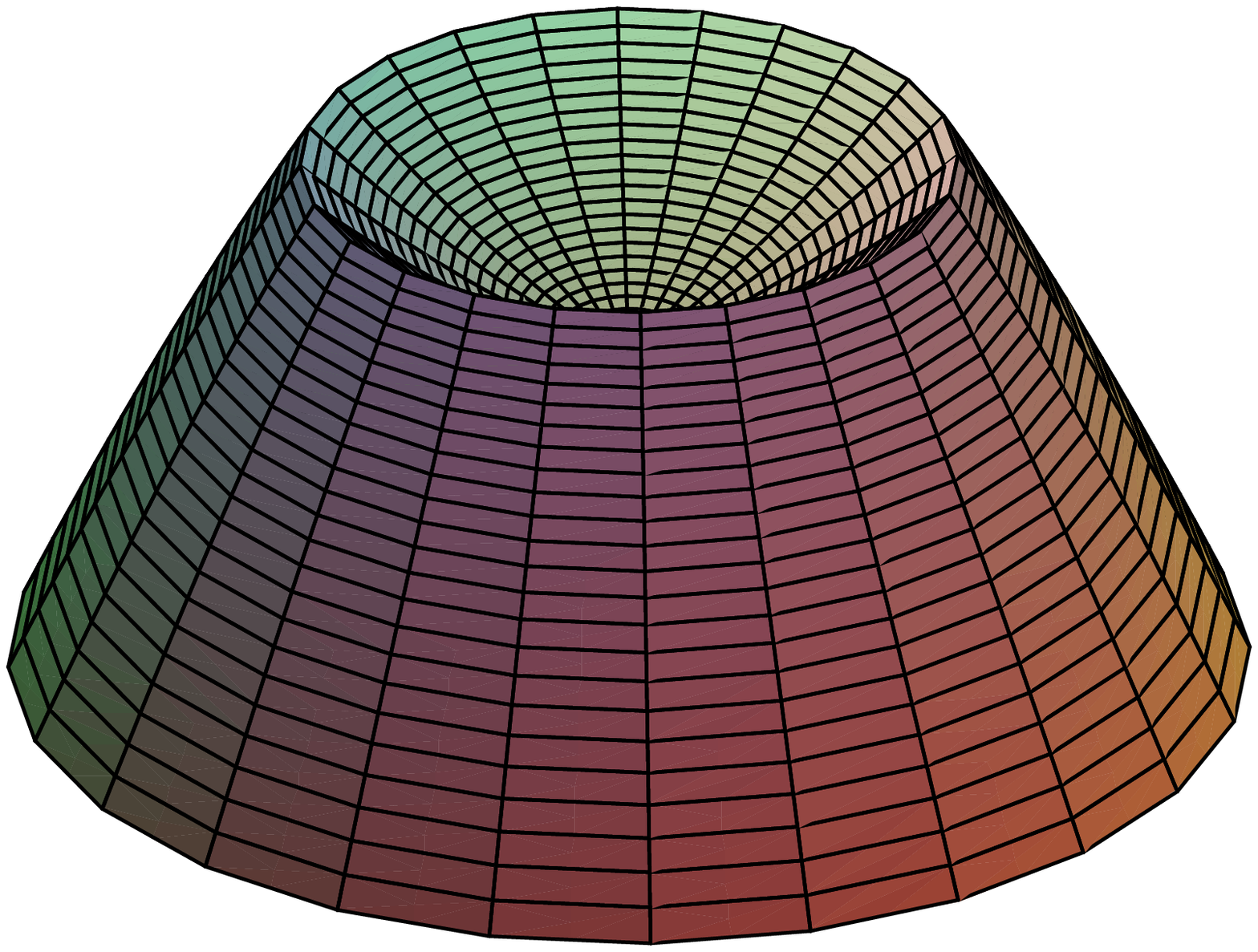}
\end{figure}

\vspace*{43mm}

\begin{figure}[h]
\rput(4.4,2.4){\includegraphics[scale=0.67]{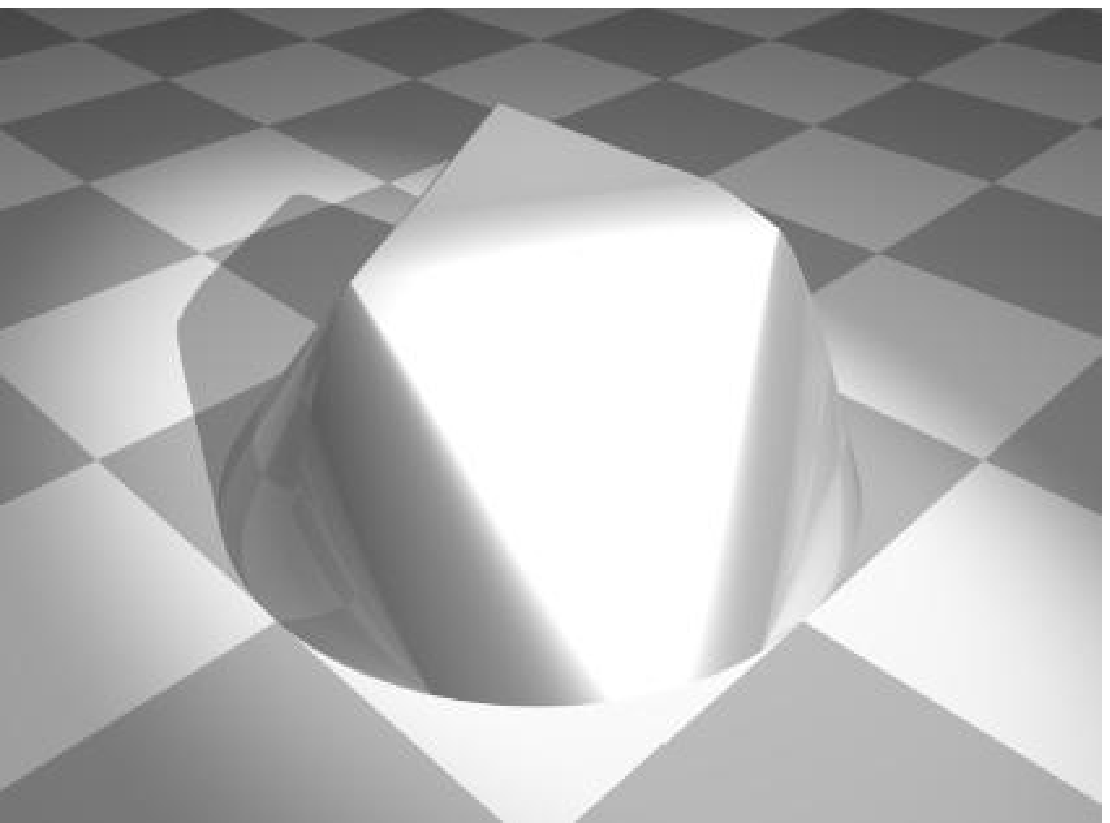}}
\rput(9.8,5){\includegraphics[scale=0.42]{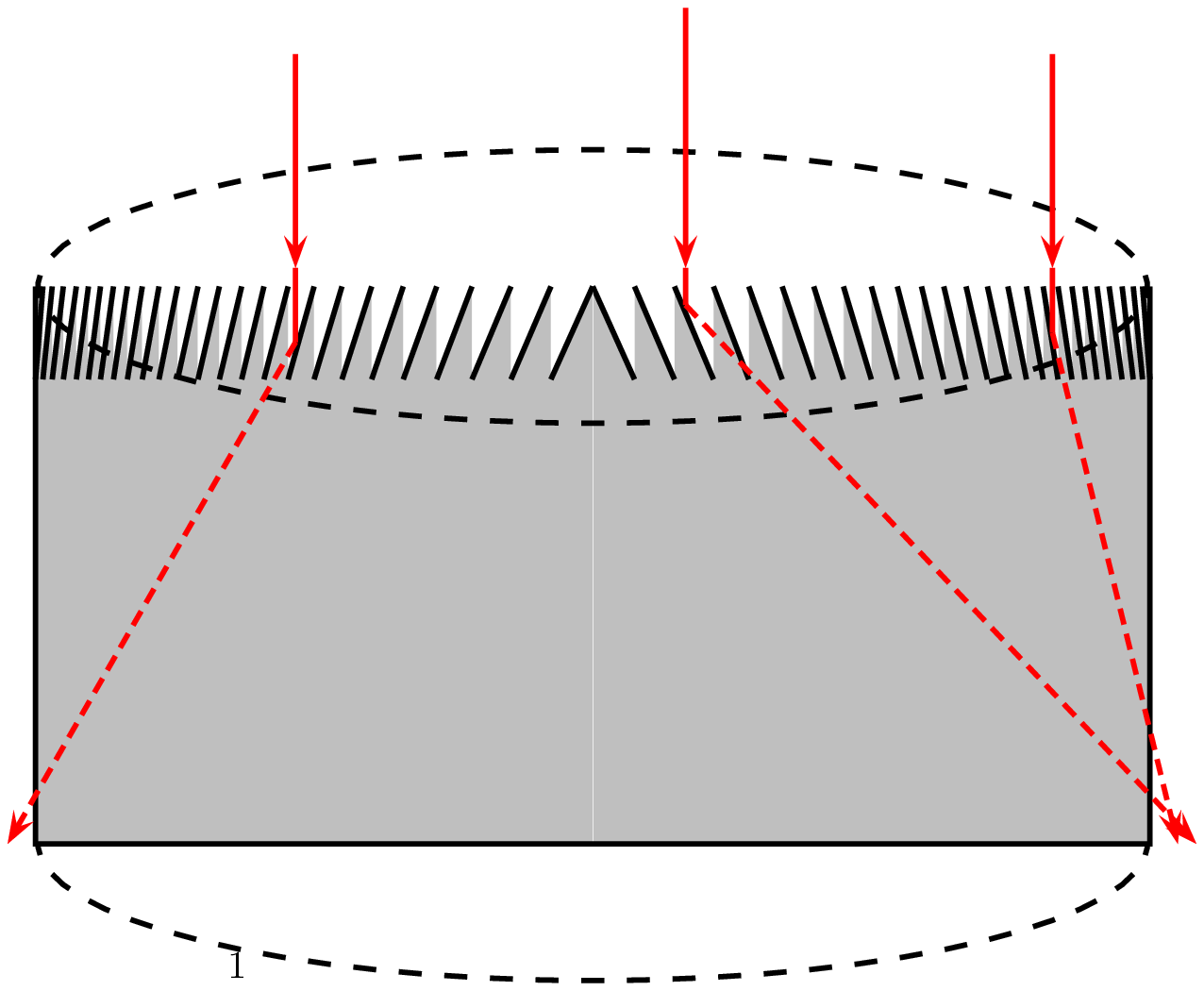}}
\caption{Optimal shapes for Problems (P$_{SC}$), (P$_{C}$), and (P$_{S}$) are shown in Figs.~(a)--(c). A schematic representation of a central vertical cross section of a nearly optimal body for Problem (P) is given in Fig.~(d).}
\label{fig4min}
\end{figure}

\rput(0,2.46){\pspolygon[linecolor=white,fillcolor=white,fillstyle=solid](-0.3,0)(7.8,0)(7.8,0.5)(-0.3,0.5)}
\rput(0,7.63){\pspolygon[linecolor=white,fillcolor=white,fillstyle=solid](-0.3,0)(7.8,0)(7.8,0.5)(-0.3,0.5)}
\rput(0.2,2.97){\pspolygon[linecolor=white,fillcolor=white,fillstyle=solid](-0.3,0)(0.5,0)(0.5,5.8)(-0.3,5.8)}
\rput(7.2,2.97){\pspolygon[linecolor=white,fillcolor=white,fillstyle=solid](-0.3,0)(0.5,0)(0.5,5.8)(-0.3,5.8)}

\rput(-1.2,0){
\psline[linewidth=0.8pt,linecolor=red,arrows=->,arrowscale=1.5](3.8,13.5)(3.8,12)
\psline[linewidth=0.8pt,linecolor=red,arrows=->,arrowscale=1.5](4,12)(4,13.5)
\psline[linewidth=0.8pt,linecolor=red,arrows=->,arrowscale=1.5](2.5,12.6)(2.5,11.35)
\psline[linewidth=0.8pt,linecolor=red,arrows=->,arrowscale=1.5](2.5,11.35)(2.5,11.3)(1.2,11.1)
\psline[linewidth=0.8pt,linecolor=red,arrows=->,arrowscale=1.5](9.8,13.5)(9.8,11.75)
\psline[linewidth=0.8pt,linecolor=red,arrows=->,arrowscale=1.5](9.8,11.75)(9.8,11.7)(13,12.2)
\psline[linewidth=0.8pt,linecolor=red,arrows=->,arrowscale=1.5](9,13)(9,11)(8,10.3)
}

\rput(0.1,9.3){(a)}
\rput(7.1,9.3){(b)}
\rput(0.2,4){(c)}
\rput(7.8,4){(d)}

The value of a nearly optimal function $u(x)$ in Problem (P) is typically close to the maximum value $M$, and $\nabla u(x)$ is typically close to
$$
-\left( \frac{M}{1 - |x|} + \sqrt{1 + \frac{M^2}{(1 - |x|)^2}} \right) \frac{x}{|x|}.
$$
On a subset of $\Om_0$ with small area we have $u(x) = 0$, and in this case $\nabla u(x) = 0$. The infimum of resistance is given by formula (\ref{o_phi}), which in our case takes the form
\beq\label{infP}
\phi(\Om_0,M) = \pi \int_0^1 \Big( 1 - \frac{M}{\sqrt{M^2 + (1 - r)^2}} \Big)\, r\, dr.
\eeq

In the following table the values of minimal resistance are provided for Problems (P$_{SC}$), (P$_{C}$), and (P) and the values $M = 0.4,\, 0.7,\, 1,\, 1.5$. The data for the first two problems are taken from \cite{LO}, and for the last one are calculated by formula (\ref{infP}).

\begin{center}
\begin{tabular}{|c|c|c|c|}
M & P$_{SC}$ & P$_{C}$ & P \\ \hline
1.5 & 0.75 & 0.70 & 0.05 \\
1 & 1.18 & 1.14 & 0.10 \\
0.7 & 1.57 & 1.55 & 0.18 \\
0.4 & 2.11 & 2.11 & 0.35 \\
\end{tabular}
\end{center}
Note that the values for Problem (P$_{C}$) were calculates with more precision in the recent paper \cite{W}. I could not find in the literature numerical values of minimal resistance concerning Problem (P$_{S}$).

As $M \to 0$, the infimum of resistance goes to $\pi$ in Problems (P$_{SC}$) and (P$_{C}$), and to $\pi/2$ in Problem (P). That is, the gain in our case is twofold.

As $M \to \infty$, the infimum in Problem (P$_{SC}$) is $\frac{27}{32} \pi M^{-2} (1 + o(1))$, and in (P) is $\frac{1}{24} \pi M^{-2} (1 + o(1))$, that is, the gain in our case is more than twentyfold, as compared with Newton's case.

This improvement seems fantastic, but it is achieved at the expense of huge complication of optimal shapes. Our method does not allow to design applicable shapes, and in this sense it merely provides a result of existence. Even to get shapes with the resistance equal to or smaller than the minimal resistance in Newton's case, one needs to use details of the construction much smaller than the size of atoms.

The plan of the rest of the paper is the following. The proof of Theorem \ref{t1} is given in Section 2, and the proofs of Theorems \ref{t2} and \ref{t3} are given in Section 3.

\section{Proof of Theorem \ref{t1}}

\subsection{The basic construction}

Take a trapezoid $AMNB$ and assume that $|MN| < |AB|$. Let $O$ be the point of intersection of the lines $AM$ and $BN$ (see Fig. \ref{fig1}).
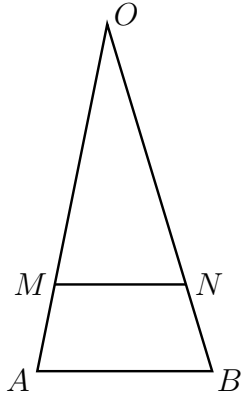
\begin{figure}[h]
\begin{picture}(0,145)
\rput(5.67,0.2){
\scalebox{1.15}{
\pspolygon(0,0)(2,0)(0.8,4)
\psline(0.2,1)(1.7,1)
\rput(-0.22,-0.1){\scalebox{0.87}{$A$}}
\rput(2.2,-0.1){\scalebox{0.87}{$B$}}
\rput(1.02,4.12){\scalebox{0.87}{$O$}}
\rput(-0.08,1){\scalebox{0.87}{$M$}}
\rput(1.97,1){\scalebox{0.87}{$N$}}
}
}
\end{picture}
\caption{An elementary pair.}
\label{fig1}
\end{figure}

Let $d = \max \{ |OA|,\, |OB| \}$,\, $d_0 =$ dist$(O,\, \trap AMNB)$,\, $h > 0$, and let $p$ be the (unique) positive value satisfying
$$
h = \frac{d^2 - p^2}{2p}.
$$
It is easy to check that
$$
p = \sqrt{d^2 + h^2} - h.
$$
For any $x \in \RRR^2$ denote by $r(x)$ the distance between $x$ and $O$, and let $\triangle$ be the open triangle $MON$, and $\trap$ be the open trapezoid $\trap AMNB$.

\begin{definition}\label{o_elem_func}\rm
The trapezoid $\trap$ is called an {\it elementary mirror} and the triangle $\triangle$, the corresponding {\it elementary valley}. The pair $(\triangle,\, \trap)$ is called an {\it elementary pair}, and the point $O$, the {\it focus of this pair}.

The {\it elementary $h$-function} $u = u_{h,{\,\trapp}} : \overline{\triangle AOB} \to \RRR$ corresponding to the elementary pair is defined by
$$
u(x) = \left\{
\begin{array}{ll}
\frac{r^2(x) - p^2}{2p}, & \text{if } x \in \trap;\\
0 & \text{otherwise.}\\
\end{array}
\right.
$$
That is, $u$ equals 0 in the triangle $\triangle$ and on the boundary of the triangle and the trapezoid $\trap$.

The {\it ratio} $\vk = \vk(\trap)$ of the elementary pair is defined by
$$
\vk = \frac{d-d_0}{d}.
$$
\end{definition}

\begin{lemma}\label{l_1}
For $x \in \trap$ one has
\beq\label{ineq1l1}
h - \vk\, (\sqrt{d^2 + h^2} + h) < u(x) < h,
\eeq
\beq\label{ineq2l1}
\frac{1}{2} \left( 1 - \frac{h}{\sqrt{d^2+h^2}} \right) < \frac{1}{1 + |\nabla u(x)|^2} < \frac{(1-\vk)^{-2}}{2} \left( 1 - \frac{h}{\sqrt{d^2+h^2}} \right).
\eeq
\end{lemma}

\begin{proof}
Since $d_0 < r(x) < d$, we have
\beq\label{eq_l1}
\frac{d_0^2 - p^2}{2p} < u(x) < \frac{d^2 - p^2}{2p}.
\eeq
The right hand side of (\ref{eq_l1}) equals $h$, and the left hand side equals
$$
\frac{d^2 - p^2}{2p} - \frac{d^2 - d_0^2}{2p} = h - \frac{d+d_0}{2} \cdot \frac{d-d_0}{\sqrt{d^2+h^2} - h}$$
$$
 > h - d\cdot \frac{d\vk}{\sqrt{d^2+h^2} - h} = h - \vk (\sqrt{d^2+h^2} + h).
$$

We have $|\nabla u(x)| = r(x)/p < d/p$, therefore
$$
\frac{1}{1 + |\nabla u(x)|^2} > \frac{p^2}{p^2 + d^2} = \frac{1}{2} \left( 1 - \frac{h}{\sqrt{d^2+h^2}} \right).
$$
On the other hand, $|\nabla u(x)| > d_0/p = (1 - \vk)d/p$, therefore
$$
\frac{1}{1 + |\nabla u(x)|^2} < \frac{p^2}{p^2 + (1 - \vk)^2 d^2} < \frac{(1 - \vk)^{-2} p^2}{p^2 + d^2} = \frac{(1-\vk)^{-2}}{2} \left( 1 - \frac{h}{\sqrt{d^2+h^2}} \right).
$$
Lemma \ref{l_1} is proved.
\end{proof}

By the first inequality in (\ref{ineq1l1}), the function $u_{h,{\,\trapp}}$ is non-negative, if
\beq\label{if}
\vk(\trap) \le \frac{1}{1 + \sqrt{1 + d^2/h^2}}.
\eeq

\begin{lemma}\label{l_2}
A nonnegative elementary $h$-function $u = u_{h,{\,\trapp}}$ satisfies SIC.
\end{lemma}

\begin{proof}
Let us give both geometrical and analytical proofs of the lemma. Geometrically, a particle that initially projects on the trapezoid, is reflected from the graph of $u$ and then goes downward along a line through the point $(O,0)$. The section of the graph of $u$ by the vertical plane containing the trajectory of the particle is shown bold in Fig. \ref{fig2}.
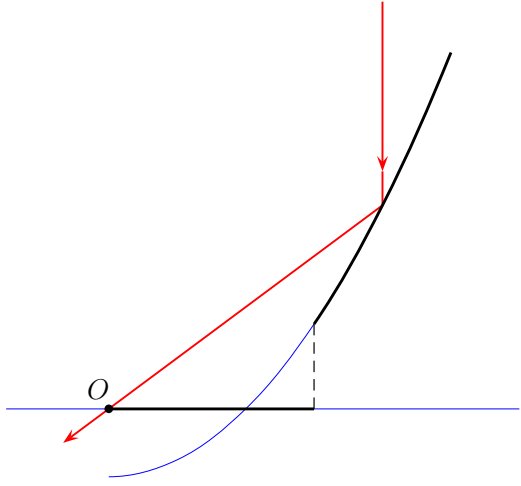
\begin{figure}[h]
\begin{picture}(0,190)
\rput(5,1){
\scalebox{0.9}{
\psline[linecolor=red,arrows=->,arrowscale=1.5](4,6)(4,3.5)
\psline[linecolor=red,arrows=->,arrowscale=1.5](4,3.5)(4,3)(-0.667,-0.5)
\psecurve[linewidth=0.3pt,linecolor=blue](-1,-0.75)(0,-1)(1,-0.75)(2,0)(3,1.25)(4,3)(5,5.25)(6,8)
\psecurve[linewidth=1.2pt](2,0)(3,1.25)(4,3)(5,5.25)(6,8)
\psline[linewidth=0.3pt,linecolor=blue](-1.5,0)(6,0)
\psline[linewidth=1.2pt](0,0)(3,0)
\psline[linewidth=0.5pt,linestyle=dashed](3,0)(3,1.25)
\psdots(0,0)
\rput(-0.15,0.3){$O$}
}
}
\end{picture}
\caption{A vertical section of graph$(u)$ (shown bold) and the trajectory of a particle.}
\label{fig2}
\end{figure}
If, on the other hand, the particle initially projects on the triangle $MON$, it is reflected vertically.

The analytical proof is a little bit more involved. Put the origin at the point $O$; then for $x \in \trap$ one has $u(x) = \frac{1}{2p}(|x|^2 - p^2)$ and $\nabla u(x) = x/p$, and condition SIC (\ref{o SIC}) reads as follows: for any $0 < t \le p$,
\beq\label{eq_l2}
u\Big(\frac{p-t}{p}\,x\Big) - u(x) \le \frac{t}{2}\, \Big(1 - \frac{|x|^2}{p^2} \Big).
\eeq
One has $\frac{p-t}{p}\,x \in \bar{\triangle} \cup \trap$. If $\frac{p-t}{p}\,x \in \bar{\triangle}$, inequality (\ref{eq_l2}) takes the form $-u(x) \le -\frac{t}{p}\, u(x)$, which is obviously true since $t \le p$ and $u(x)> 0$. If $\frac{p-t}{p}\,x \in \trap$, inequality (\ref{eq_l2}) takes the form
$$
\frac{\big(\frac{p-t}{p}\big)^2 |x|^2 - |x|^2}{2p} \le \frac{t}{2} \left( 1 - \frac{|x|^2}{p^2} \right),
$$
which after some algebra reduces to the obvious inequality $(t - p)\, |x|^2/p^3 \le 1$.

If $x \in \triangle$, one has $u(x) = 0, \ \nabla u(x) = 0$, and $u\big(\frac{p-t}{p}\,x\big) = 0$ for all $0 < t \le p$, and condition SIC takes the form $0 \le 1/2$. Lemma \ref{l_2} is proved.
\end{proof}

The resistance of a nonnegative function $u_{h,{\,\trapp}}$ equals
\beq\label{res_elem}
F(u_{h,{\,\trapp}}) = |\triangle| + \int_{\trapp} \, \frac{dx}{1 + |\nabla u(x)|^2};
\eeq
here and below, $|\triangle|$ and $|\trap|$ mean the areas of the triangle $MON$ and the trapezoid $AMNB$, respectively. By the second inequality in (\ref{ineq2l1}), the integrand in the right hand side of (\ref{res_elem}) does not exceed
$\frac 12 (1-\vk)^{-2} \left(1 - {h}/{\sqrt{d^2 + h^2}} \right),$
so the resistance of $u$ is estimated as follows:
$$
F(u_{h,{\,\trapp}}) \le |\triangle| + \frac{(1-\vk)^{-2}}{2} \left(1 - \frac{h}{\sqrt{d^2 + h^2}} \right) \cdot |\trap|.
$$

\begin{lemma}\label{l_3}
Consider a finite collection of elementary pairs ${\triangle_i},{\,\trap_i}$ with poles at $O_i$. Choose positive values $h_i$ so that the elementary $h_i$-functions $u_i = u_{h_i,{\,\trapp_i}} : \overline{\triangle_i \cup \trap_i} \to \RRR$ are nonnegative. Let a convex domain $\Om$ be such that $\Om \subset \cup_i \overline{(\triangle_i \cup \trap_i)}$ and all $O_i$ lie outside $\Om$. Let the function $u : \bar{\Om} \to \RRR$ be defined by $u(x) = \min_i u_i(x)$, where the infimum is taken over those $i$ for which $x \in \overline{\triangle_i \cup \trap_i}$. Then $u$ satisfies SIC.
\end{lemma}

The proof of Lemma \ref{l_3} is a simple consequence of the definitions and is left to the reader.

The following important lemma will be proved in the next two subsections.

\begin{lemma}\label{l_main}
For any $\ve > 0$ there exist a finite family of elementary pairs $\triangle_i,\, \trap_i$ with ratios $\vk_i$ and with foci at $O_i$ such that

(A) $\Om \subset \cup_i \overline{(\triangle_i \cup \trap_i)}$;

(B) $|\cup_i \triangle_i| < \ve$;

(C) $\vk_i < \ve$;

(D) for each $i$, $O_i \not\in \Om$;

(E) for each $i$ and $x \in \trap_i$,\, $|x - O_i| < \text{dist}(x, \pl\Om) + \ve$.
\end{lemma}

The set
$$
V = (\cup_i \triangle_i) \cap \Om
$$
is called {\it the valley} of the family, and its complement $\Om \setminus V,$ {\it the mirror} of the family.

Notice that the family of elementary pairs $\triangle_i = \triangle_i(\ve), \ \trap_i = \trap_i(\ve)$ indicated in this lemma depends on $\ve$.

Let us now derive Theorem \ref{t1} from Lemma \ref{l_main}. Consider the elementary $M$-functions $u_{i,\ve} = u_{M,\,\trapp_i(\ve)} : \overline{\triangle_i \cup \trap_i} \to \RRR$ and define the function $u_\ve : \bar{\Om} \to \RRR$ by $u_\ve(x) = \inf_i u_{i,\ve}(x)$, where the infimum is taken over those $i$ for which $x \in \overline{\triangle_i \cup \trap_i}$.

\begin{lemma}\label{l_4}
For $\ve$ sufficiently small we have $u_{\ve} \in S_{\Om,M}$ and $F(u_{\ve}) < \phi(\Om,M) + O(\ve), \ \ve \to 0.$ Further, there exists a finite set of points $O_i = O_i^\ve \not\in \Om$ in the $\ve$-neighborhood of $\Om$ and a domain $V \subset \Om$ with area $|V| < \ve$ such that each incident particle corresponding to $x \in V$ is reflected vertically, and each particle corresponding to a regular point $x \in \Om \setminus V$ after the reflection passes through one of the points $(O_i, 0)$ on the $x$-plane.
\end{lemma}


\begin{proof}
Property (E) implies that $d_i \le \text{diam}(\Om) + \ve$. Taking into account property (C), we conclude that for $\ve$ sufficiently small and arbitrary $i$,
$$
\vk_i < \ve < \frac{1}{1 + \sqrt{1 + d_i^2/M^2}},
$$
and so, inequality (\ref{if}) is satisfied for the function $u_{i,\ve}$. This means that $0 \le u_{i,\ve} \le M$, and by Lemma \ref{l_2}, $u_{i,\ve}$ satisfies SIC. By properties (A) and (D) and Lemma \ref{l_3}, $u_{\ve}$ also satisfies SIC. Therefore $u_\ve \in S_{\Om,M}$.

If $x \in V$, the particle is reflected vertically. If $x$ is a regular point of $\Om \setminus V$ then for some $i$ we have the equality $u_\ve = u_{i,\ve}$ in a neighborhood of $x$, and so, $x \in \trap_i$. The corresponding particle after the reflection passes through the point $(O_i, 0)$ on the $x$-plane. Using (D) and (E) and taking into account that $x \in \trap \cap \Om$, one easily concludes that dist$(O_i, \Om) < \ve.$ Indeed, the segment $[x, O_i]$ contains a point $y \in \pl\Om,$ and since $|x - y| \ge$ dist$(x, \pl\Om)$ and $|x - O_i| <$dist$(x, \pl\Om) + \ve$, we have
$$
\text{dist}(O_i, \Om) \le |O_i - y| = |x - O_i| - |x - y| < \ve.
$$

Further, by Lemma \ref{l_1}
$$
\frac{1}{1 + |\nabla u_\ve(x)|^2} = \frac{1}{1 + |\nabla u_{i,\ve}(x)|^2} < \frac{(1-\vk_i)^{-2}}{2} \left( 1 - \frac{M}{\sqrt{M^2 + d_i^2}} \right)
$$
$$
< \frac{(1-\ve)^{-2}}{2} \left( 1 - \frac{M}{\sqrt{M^2 + (d(x) + \ve)^2}} \right)
$$
(recall that $d(x) = \text{dist}(x,\pl\Om)$). Thus,
$$
F(u_\ve) = \int_\Om \frac{dx}{1 + |\nabla u_\ve(x)|^2} < |V| + \frac{(1-\ve)^{-2}}{2} \int_{\Om \setminus V} \left( 1 - \frac{M}{\sqrt{M^2 + (d(x) + \ve)^2}} \right) dx.
$$
By property (B), $|V| < \ve$, and recalling equation (\ref{o_phi}) we conclude that $F(u) < \phi(\Om,M) + O(\ve)$. Lemma \ref{l_4} is proved.
\end{proof}

The claim of Theorem \ref{t1} is a consequence of Lemma \ref{l_4} and equation (\ref{ineq1}).

\subsection{Proof of  Lemma \ref{l_main}}\label{subs22}

It is assumed that we are given a bounded convex open set $\Om$ and a positive value $\ve$.

Consider a square lattice $\del\ZZZ \times \del\ZZZ$, the size $\del$ of the lattice to be specified below, and take the (closed) squares $Q_1, \ldots, Q_N$ of the lattice that have nonempty intersection with $\Om$. Consider also the lattice $2\del\ZZZ \times 2\del\ZZZ$ with double size and take the squares $\tilde Q_1,\, \tilde Q_2, \ldots$ of this lattice that do not intersect $\Om$, $\tilde Q_i \cap \Om = \emptyset.$ For each square $Q_i$ find the square $\tilde Q_i$ such that the distance between their centers is minimal (and change the numeration of $\tilde Q_i$ if necessary); see Fig.~\ref{fig22}.

\begin{figure}[h]
\begin{picture}(0,145)
\rput(5,1.1){
\scalebox{1}{
\psecurve[fillcolor=lightgray,fillstyle=solid](0,2)(0,0)(3.7,-0.3)(3.7,2.3)(0,2)(0,0)(3.7,-0.3)
   \rput(-0.17,-0.09){
\psline[linewidth=0.6pt](-1,2)(-1,0)
\psline[linewidth=0.4pt,linecolor=blue,linestyle=dashed](-0.5,2)(-0.5,0)
\psline[linewidth=0.6pt](0,3)(0,-1)
\psline[linewidth=0.4pt,linecolor=blue,linestyle=dashed](0.5,3)(0.5,-1)
\psline[linewidth=0.6pt](1,3)(1,-1)
\psline[linewidth=0.4pt,linecolor=blue,linestyle=dashed](1.5,3)(1.5,-1)
\psline[linewidth=0.6pt](2,3)(2,-1)
\psline[linewidth=0.4pt,linecolor=blue,linestyle=dashed](2.5,3)(2.5,-1)
\psline[linewidth=0.6pt](3,3)(3,-1)
\psline[linewidth=0.4pt,linecolor=blue,linestyle=dashed](3.5,3)(3.5,-1)
\psline[linewidth=0.6pt](4,3)(4,-1)
\psline[linewidth=0.4pt,linecolor=blue,linestyle=dashed](4.5,3)(4.5,0)
\psline[linewidth=0.6pt](5,3)(5,0)
\psline[linewidth=0.6pt](0,-1)(4,-1)
\psline[linewidth=0.4pt,linecolor=blue,linestyle=dashed](0,-0.5)(4,-0.5)
\psline[linewidth=0.6pt](-1,0)(5,0)
\psline[linewidth=0.4pt,linecolor=blue,linestyle=dashed](-1,0.5)(5,0.5)
\psline[linewidth=0.6pt](-1,1)(5,1)
\psline[linewidth=0.4pt,linecolor=blue,linestyle=dashed](-1,1.5)(5,1.5)
\psline[linewidth=0.6pt](-1,2)(5,2)
\psline[linewidth=0.4pt,linecolor=blue,linestyle=dashed](0,2.5)(5,2.5)
\psline[linewidth=0.6pt](0,3)(5,3)
\pspolygon[fillstyle=solid,fillcolor=gray](3,2)(2.5,2)(2.5,1.5)(3,1.5)
\pspolygon[fillstyle=solid,fillcolor=gray](3,3)(3,4)(2,4)(2,3)
\psline[arrows=->,arrowscale=2](2.75,1.75)(2.5,3.5)
\pspolygon[fillstyle=solid,fillcolor=gray](1,1)(0.5,1)(0.5,0.5)(1,0.5)
\pspolygon[fillstyle=solid,fillcolor=gray](0,0)(0,-1)(-1,-1)(-1,0)
\psline[arrows=->,arrowscale=2](0.75,0.75)(-0.5,-0.5)
\rput(3.25,1.75){$Q_i$}
\rput(3.3,3.5){$\tilde Q_i$}
\rput(1.25,0.75){$Q_j$}
\rput(-1.3,-0.5){$\tilde Q_j$}
}
}
}
\end{picture}
\caption{Covering of $\Om$ by squares of the lattice $\del\ZZZ \times \del\ZZZ$ and two cases of correspondence between the squares $Q_i$ and $\tilde Q_i$.}
\label{fig22}
\end{figure}
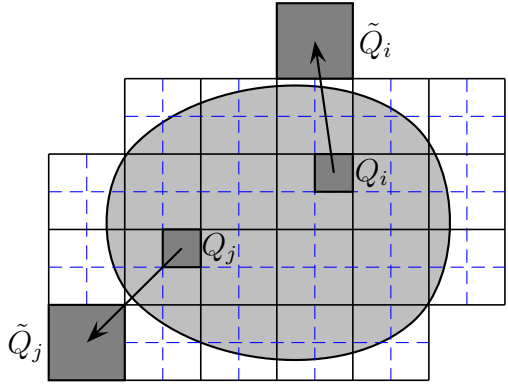

The correspondence $Q_i \mapsto \tilde Q_i$ is not necessarily injective; that is, it may happen that $\tilde Q_i = \tilde Q_j$ for $i \ne j$.

Choose $\del$ small enough so that for each $i$ and for any two points $x \in Q_i, \ \tilde x \in \tilde Q_i$,
\beq\label{ineqf}
|x - \tilde x| < \text{dist}(x,\pl\Om) + \ve/2.
\eeq
It suffices to take, for example, $\del = \ve/10$.

The following lemma will be proved in the next subsection. We use the notation $B_\om(\CCC)$ for the ball with radius $\om$ centered at the point $\CCC$.

\begin{lemma}\label{l2main}
For any triangle $\AAA\BBB\CCC$ there exists a finite family of elementary pairs $\triangle_i^\om,\, \trap_i^\om$ depending on the parameter $\om > 0$ with ratios $\vk_i^\om$ and with foci at $O_i^\om$ (the number of pairs in the family may depend on $\om$) such that

(a) $B_\om(\CCC) \subset \cup_i \overline{(\triangle_i^\om \cup \trap_i^\om)}$;

(b) $|\cup_i \triangle_i^\om|/\om^2 \to 0$ as $\om \to 0$;

(c) $\max_i \vk_i^\om \to 0$ as $\om \to 0$;

(d) the foci $O_i^\om$ belongs to the $\al(\om)$-neighborhood of the segment $\AAA\BBB$, and

(e) $\cup_i \trap_i^\om \subset B_{\al(\om)}(C)$, where $\al(\om) \to 0$ as $\om \to 0$.
\end{lemma}

Let us prove Lemma \ref{l_main} using this lemma.

Fix $i$ and define a triangle $\AAA\BBB\CCC$ so that the translates of $Q_i$ by the vectors $\overrightarrow{\CCC\AAA}$ and $\overrightarrow{\CCC\BBB}$ lie in the interior of $\tilde Q_i$. This is possible, since the size of $Q_i$ is smaller than that of $\tilde Q_i$.

Take an $n \in \mathbb{N}$ and divide $Q_i$ into $n^2$ small squares; let $\CCC_1, \ldots, \CCC_{n^2}$ be their centers (see Fig.~\ref{fig_l5}\,(b)). The size of each square is
\beq\label{oom}
\om = \text{size}(Q_i)/n.
\eeq
Take the family of elementary pairs $\{ \triangle_i^\om,\, \trap_i^\om \}$ defined by Lemma \ref{l2main} for the triangle $\AAA\BBB\CCC$, and for each $k = 1, \ldots, n^2$ consider the  translate of this family by the vector $\overrightarrow{\CCC\CCC_k}$ (see Fig.~\ref{fig_l5}\,(b)). The corresponding translate of the triangle $\AAA\BBB\CCC$ will be denoted by $\AAA_k\BBB_k\CCC_k$, and the union of translates of the family $\{ \triangle_i^\om,\, \trap_i^\om \}$ will be referred to as the {\it big family corresponding to} $Q_i$. We have the following.

\begin{figure}[h]
\begin{picture}(0,145)
\rput(3,1){
\scalebox{1}{
\pscircle[fillstyle=solid,fillcolor=lightgray](-0.5,0){0.3}
\psdots(0,2.5)(-0.5,0)(1,2.5)
\psline[linecolor=blue](0,2.25)(1,2.25)
\psline[linecolor=blue](0,2.75)(1,2.75)
\psarc[linecolor=blue](1,2.5){0.25}{-90}{90}
\psarc[linecolor=blue](0,2.5){0.25}{90}{270}
\psline[linestyle=dashed](0,2.5)(-0.5,0)(1,2.5)
\psline(0,2.5)(1,2.5)
\rput(0.2,0){\scalebox{1}{$\CCC$}}
\rput(1.39,2.82){\scalebox{1}{$\AAA$}}
\rput(-0.35,2.8){\scalebox{1}{$\BBB$}}
\pscircle[linecolor=blue,linestyle=dashed](-0.5,0){0.5}
\rput(-2,-0.9){\scalebox{1}{(a)}}
}
}

\rput(10,1){
\scalebox{0.8}{
\pspolygon(-1,-1)(1,-1)(1,1)(-1,1)
\psline(-0.5,-1)(-0.5,1)
\psline(0,-1)(0,1)
\psline(0.5,-1)(0.5,1)
\psline(-1,-0.5)(1,-0.5)
\psline(-1,0)(1,0)
\psline(-1,0.5)(1,0.5)
\rput(0,0.5){\pspolygon(-1.25,1)(2.25,1)(2.25,4.5)(-1.25,4.5)}
\rput(-1.8,3.1){\scalebox{1.7}{$\tilde Q_i$}}
\rput(-1.55,-0.2){\scalebox{1.45}{$Q_i$}}
\rput(0.18,0.28){
\psline[linecolor=blue,linestyle=dashed](0,2.25)(1,2.25)
\psline[linecolor=blue,linestyle=dashed](0,2.75)(1,2.75)
\psarc[linecolor=blue,linestyle=dashed](1,2.5){0.25}{-90}{90}
\psarc[linecolor=blue,linestyle=dashed](0,2.5){0.25}{90}{270}
\psline[linestyle=dashed](0,2.5)(-0.5,0)(1,2.5)
\psline(0,2.5)(1,2.5)
\psdots(0,2.5)(-0.5,0)(1,2.5)
\rput(0,-0.04){\scalebox{1.12}{$\CCC_k$}}
\rput(1.37,2.82){\scalebox{1.12}{$\AAA_k$}}
\rput(-0.4,2.84){\scalebox{1.12}{$\BBB_k$}}
}
\rput(-2.75,-1){\scalebox{1.25}{(b)}}
}
}
\end{picture}
\caption{(a) The union of elementary sets in the family contains the shaded circle, and all the trapezoids are contained in the circle bounded by a dashed line. The focal set of the family is contained in the neighborhood of $\AAA\BBB$ shown by a dashed line. (b) A pair of squares $Q_i,\, \tilde Q_i$ and a translate of $\triangle \AAA\BBB\CCC$ corresponding to a small square in $Q_i$.}
\label{fig_l5}
\end{figure}
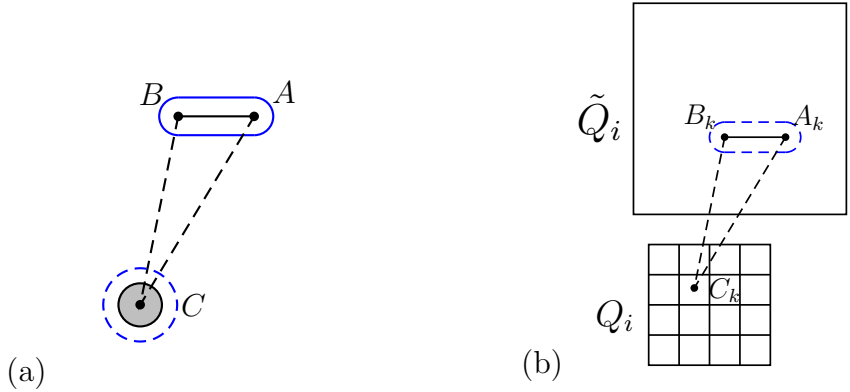

(a) The union of elementary sets in the $k$th family contains the $k$th small square; therefore the union of elementary sets in the big family contains $Q_i.$

(b) By (\ref{oom}), the area of the union of elementary triangles in the big family is not greater than
$$
n^2 |\cup_i \triangle_i^\om| = (\text{side}(Q_i))^2 \frac{1}{\om^2} |\cup_i \triangle_i^\om| \to 0 \quad \text{as} \ \, \om \to 0.
$$
For $n$ sufficiently large it is smaller than $\ve/N$.

(c) For $n$ sufficiently large (and therefore, $\om$ sufficiently small), $\vk_i^\om < \ve.$

(d) For $\om$ sufficiently small, not only the translates of $Q_i$ by $\overrightarrow{\CCC\BBB}$ and $\overrightarrow{\CCC\AAA}$ belong to $\tilde Q_i$, but also their $\al(\om)$-neighborhoods. This implies that the $\al(\om)$-neighborhoods of the segments $\overrightarrow{\AAA_k\BBB_k}, \ k = 1, \ldots, n^2$ also belong to $\tilde Q_i$. Thus, for $n$ sufficiently large (and correspondingly $\om$ sufficiently small), all foci of the big family belong to $\tilde Q_i$, and therefore do not belong to $\Om.$

(e) For $\om$ sufficiently small, $\al(\om) < \ve/4$, and therefore the union of the trapezoids of the big family belongs to the $(\ve/4)$-neighborhood of $Q_i$. Take a point $x$ in a trapezoid of the big family and let $O$ be the corresponding focus. Then there exists a point $x' \in Q_i$ such that $|x-x'| < \ve/4$ and therefore
$$
\text{dist}(x', \pl\Om) < \text{dist}(x, \pl\Om) + \ve/4.
$$
Since (for $\om$ sufficiently small) $O \in \tilde Q_i$, by (\ref{ineqf}) we have
$$
|x' - O| < \text{dist}(x',\pl\Om) + \ve/2,
$$
and thus,
$$
|x - O| < |x - x'| + |x' - O| < \ve/4 + \text{dist}(x', \pl\Om) + \ve/2 < \text{dist}(x, \pl\Om) + \ve.
$$

Take the union of the big families corresponding to all $Q_i, \ i = 1, \ldots, n^2$; it follows from (a)--(e) that it satisfies the conditions (A)--(E). Lemma \ref{l_main} is proved.

\subsection{Proof of Lemma \ref{l2main}}\label{subs23}

The family to be constructed is a two-level hierarchy. Families of the 1st order are constructed by the so-called procedure of $\omom$-doubling of a triangle. This procedure was first proposed in \cite{hollows} to solve a problem of minimal resistance for cavities and is inspired by Besicovitch's method of solving the Kakeya problem \cite{Bes}. For the reader's convenience, the procedure is described here.

Let $0 < \omom < 1$. Take a triangle $MON$ and extend the sides $OM$ and $ON$ beyond the point $O$ to obtain the segments $MM'$ and $NN'$, with
$$
|OM'| = \omom|OM| \quad \text{and} \quad |ON'| = \omom|ON|.
$$
Joining the points $M'$ and $N'$ with the midpoint $D$ of $MN$, we obtain two triangles $MM'D$ and $NN'D$ (see Fig.~\ref{fig_doubling}). The procedure of $\omom$-doubling consists in substituting the original triangle $MON$ with the two triangles $MM'D$ and $NN'D$. If the height of $\triangle MON$ is $h$, then the heights of the new triangles are both equal to $(1 + \omom)h$.

\begin{figure}[h]
\begin{picture}(0,200)
\rput(5.7,-0.6){
\pspolygon(0.2,1)(1,5)(2.6,1)
\psline[linewidth=0.4pt,linecolor=blue,linestyle=dashed](0.2,7)(1.4,7)
  \psline[linewidth=0.8pt](1,5)(1.4,7)(1.4,1)
  \psline[linewidth=0.8pt](1.4,1)(0.2,7)(1,5)
\psline[linewidth=0.4pt,linestyle=dashed,linecolor=blue](0.2,7)(1.4,7)
\psline[linewidth=0.4pt,linecolor=blue,linestyle=dashed](0.2,7)(0.6,3)
\psline[linewidth=0.4pt,linecolor=blue,linestyle=dashed](1.4,7)(1.8,3)
\rput(1.2,5.07){\scalebox{0.82}{$O$}}
\rput(-0.03,1.1){$M$}
\rput(2.8,1.16){$N$}
\rput(-0.07,7){\scalebox{0.9}{$N'$}}
\rput(1.68,7){\scalebox{0.82}{$M'$}}
   \rput(0.25,2.95){\scalebox{0.9}{$N''$}}
   \rput(2.1,3.1){\scalebox{0.82}{$M''$}}
\rput(1.3,0.73){\scalebox{0.9}{$D$}}
   \psline[linewidth=0.4pt,linecolor=blue,linestyle=dashed](1,5)(1.4,1)
}
\end{picture}
\caption{The procedure of doubling: the triangle $MON$ is substituted with the two triangles $MM'D$ and $NN'D$.}
\label{fig_doubling}
\end{figure}
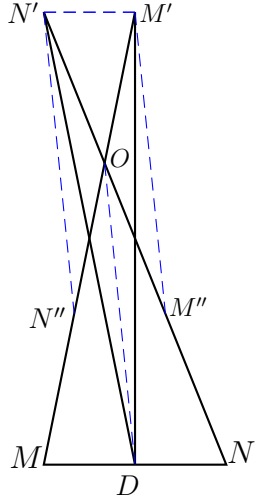

It is easy to estimate the increase of the total area, $|\triangle MM'D \cup \triangle NN'D| - |\triangle MON|$, as a result of doubling. Draw two lines parallel to $OD$ through $M'$ an $N'$, and denote by $M''$ and $N''$ the points of intersection of these lines with $ON$ and $OM$, respectively. We have
$$
|\triangle ON'N''| = |\triangle OM'M''| = \omom^2 |\triangle MON|,
$$
therefore
$$
|\triangle MM'D \cup \triangle NN'D| - |\triangle MON| < |\triangle ON'N'' \cup \triangle OM'M''| = 2\omom^2 |\triangle MON|.
$$

Let us now apply the procedure of doubling successively several times, starting from the triangle $MON$. Assume that the height of $MON$ is 1 and the length of the base is $d$. At the $m$th step, $m = 1,\, 2,\ldots$ we apply $\omom_m$-doubling with $\omom_m = 1/m$. After $m$ steps we get $2^m$ triangles with height $m+1$ and with the length of base $2^{-m} d$.

Let $S_m$ be the area of the union of triangles after the $m$th step. One has $S_0 = |\triangle MON| = d/2$, and the total increase of the area at the $m$th step is smaller than $2^{m-1} \cdot 2\del_m^2 \cdot 2^{-(m-1)} dm/2 = d/m$, therefore
$$
S_m < S_{m-1} + d/m.
$$
One easily concludes by induction that $S_m < d(\ln m + 3/2)$ for $m \ge 1$.

Now fix $l > 0$ and for $m = 1,\, 2,\ldots$ define the family of $2^m$ elementary pairs ${\triangle_i},{\,\trap_i}$, where the triangles $\triangle_i$ coincide with the triangles obtained at the $m$th step of doubling, and all the trapezoids $\trap_i$ have the same height $l$. Each trapezoid of the $(m-1)$th step generates two trapezoids of the $m$th step, which are disjoint and contained in the original one. This implies that for each $m$, the trapezoids of the $m$th step are mutually disjoint. The area of the union of trapezoids is greater than $ld$.

\begin{definition}\label{o_family1step}\rm
A family ${\triangle_i},{\,\trap_i}$ of $2^m$ elementary pairs at the $m$th step with $l = \sqrt m$ is called a {\it family of the 1st order}. The union $\cup_{i=1}^{2^m} \overline{(\triangle_i \cup \trap_i)}$ is called a $(1,m)$-{\it set}.
The (finite) set of foci of the elementary pairs is called the corresponding {\it focal set of the 1st order}.
The original triangle $MON$ is called the {\it generating triangle} for the family of the 1st order.
\end{definition}

Several families of the 1st order with $m = 0,\, 1,\, 2,\, 3,\, 4$ are shown in Fig.~\ref{fig_set1order}.

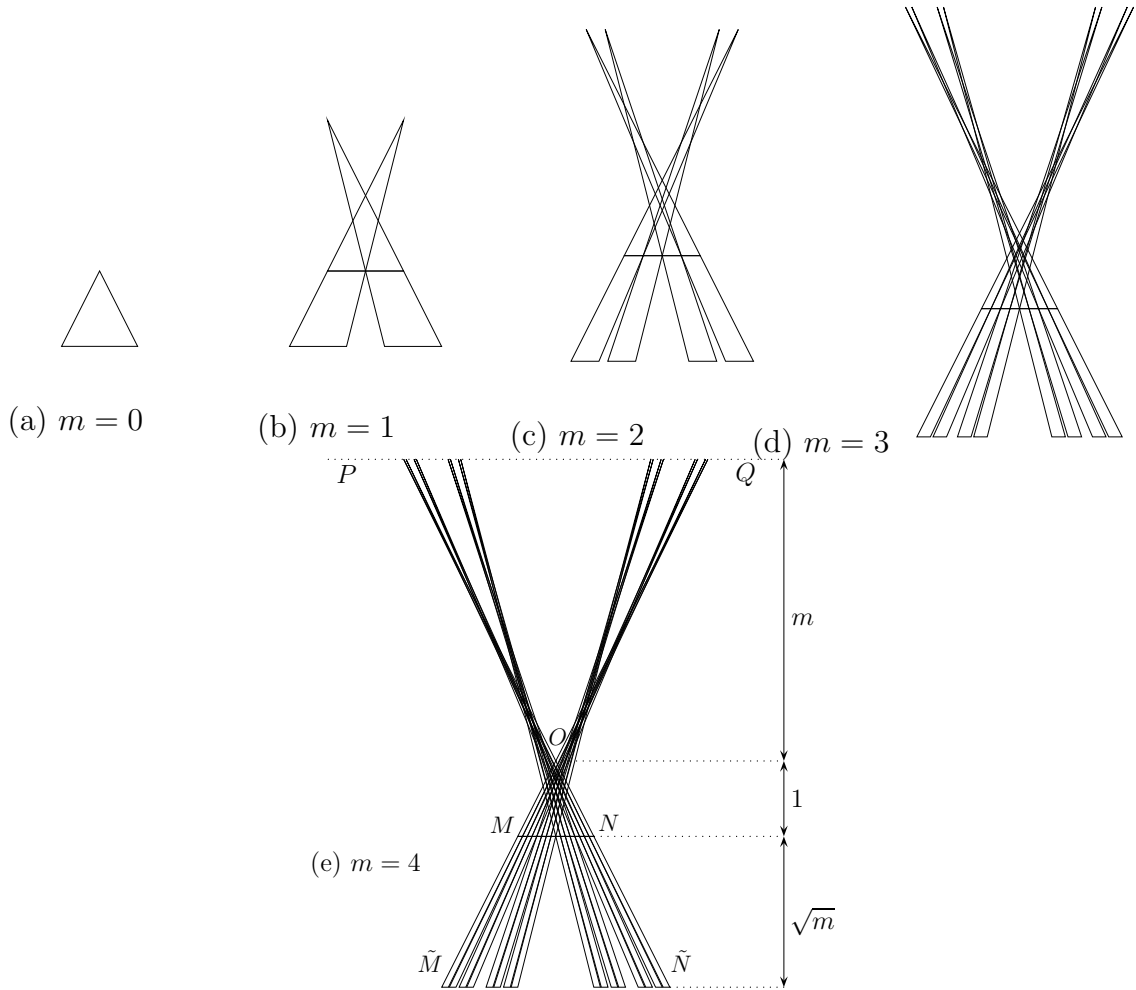
\begin{figure}[h]
\begin{picture}(0,355)

\rput(0,8.8){

\rput(0.5,-0.5){
\scalebox{0.25}{
\pspolygon(-2,0)(2,0)(0,4)
}
}

\rput(4,0.5){
\scalebox{0.25}{
\pspolygon(-2,0)(2,8)(0,0)
\pspolygon(2,0)(-2,8)(0,0)
\pspolygon(2,0)(4,-4)(1,-4)(0,0)
\pspolygon(-2,0)(-4,-4)(-1,-4)(0,0)
}
}

\rput(7.9,0.7){
\scalebox{0.25}{
\pspolygon(-2,0)(4,12)(-1,0)
\pspolygon(-1,0)(3,12)(0,0)
\pspolygon(2,0)(-4,12)(1,0)
\pspolygon(1,0)(-3,12)(0,0)
  \pspolygon(-2,0)(-4.8,-5.6)(-3.33,-5.6)(-1,0)
  \pspolygon(-1,0)(-2.866,-5.6)(-1.4,-5.6)(0,0)
  \pspolygon(2,0)(4.8,-5.6)(3.33,-5.6)(1,0)
  \pspolygon(1,0)(2.866,-5.6)(1.4,-5.6)(0,0)
}
}

\rput(12.6,0){
\scalebox{0.25}{
\pspolygon(2,0)(-6,16)(1.5,0) 
\pspolygon(1.5,0)(-5.667,16)(1,0) 
\pspolygon(-1,0)(4.333,16)(-0.5,0)
\pspolygon(-0.5,0)(4,16)(0,0)
\pspolygon(1,0)(-4.333,16)(0.5,0) 
\pspolygon(0.5,0)(-4,16)(0,0) 
\pspolygon(-2,0)(6,16)(-1.5,0)
\pspolygon(-1.5,0)(5.667,16)(-1,0)
  \pspolygon(5.4,-6.8)(2,0)(1.5,0)(4.69,-6.8)
  \pspolygon(-5.4,-6.8)(-2,0)(-1.5,0)(-4.69,-6.8)
  \pspolygon(4.55,-6.8)(1.5,0)(1,0)(3.84,-6.8)
  \pspolygon(-4.55,-6.8)(-1.5,0)(-1,0)(-3.84,-6.8)
  \pspolygon(-3.266,-6.8)(-1,0)(-0.5,0)(-2.55,-6.8)
  \pspolygon(3.266,-6.8)(1,0)(0.5,0)(2.55,-6.8)
  \pspolygon(-2.411,-6.8)(-0.5,0)(0,0)(-1.7,-6.8)
  \pspolygon(2.411,-6.8)(0.5,0)(0,0)(1.7,-6.8)
}
}

\rput(0.2,-1.5){(a) $m=0$}
\rput(3.5,-1.6){(b) $m=1$}
\rput(6.8,-1.7){(c) $m=2$}
\rput(10,-1.8){(d) $m=3$}
}

\rput(6.5,1.8){
\scalebox{0.25}{
\pspolygon(2,0)(-8,20)(1.75,0)
\pspolygon(1.5,0)(-7.875,20)(1.75,0)
\pspolygon(1.5,0)(-7.4583,20)(1.25,0)
\pspolygon(1,0)(-7.3333,20)(1.25,0)
\pspolygon(1,0)(-5.667,20)(0.75,0)
\pspolygon(0.5,0)(-5.542,20)(0.75,0)
\pspolygon(0.5,0)(-5.125,20)(0.25,0)
\pspolygon(0,0)(-5,20)(0.25,0)
\pspolygon(-2,0)(8,20)(-1.75,0)
\pspolygon(-1.5,0)(7.875,20)(-1.75,0)
\pspolygon(-1.5,0)(7.4583,20)(-1.25,0)
\pspolygon(-1,0)(7.3333,20)(-1.25,0)
\pspolygon(-1,0)(5.667,20)(-0.75,0)
\pspolygon(-0.5,0)(5.542,20)(-0.75,0)
\pspolygon(-0.5,0)(5.125,20)(-0.25,0)
\pspolygon(0,0)(5,20)(-0.25,0)
\pspolygon(6,-8)(2,0)(1.5,0)(5.25,-8)
  \psline(1.75,0)(5.65,-8)
  \psline(1.75,0)(5.6,-8)
\pspolygon(-6,-8)(-2,0)(-1.5,0)(-5.25,-8)
  \psline(-1.75,0)(-5.65,-8)
  \psline(-1.75,0)(-5.6,-8)
\pspolygon(5.084,-8)(1.5,0)(1,0)(4.336,-8)
  \psline(1.25,0)(4.682,-8)
  \psline(1.25,0)(4.734,-8)
\pspolygon(-5.084,-8)(-1.5,0)(-1,0)(-4.336,-8)
  \psline(-1.25,0)(-4.682,-8)
  \psline(-1.25,0)(-4.734,-8)
\pspolygon(-3.667,-8)(-1,0)(-0.5,0)(-2.918,-8)
  \psline(-0.75,0)(-3.31,-8)
  \psline(-0.75,0)(-3.262,-8)
\pspolygon(3.667,-8)(1,0)(0.5,0)(2.918,-8)
  \psline(0.75,0)(3.31,-8)
  \psline(0.75,0)(3.262,-8)
\pspolygon(-2.748,-8)(-0.5,0)(0,0)(-2,-8)
  \psline(-0.25,0)(-2.394,-8)
  \psline(-0.25,0)(-2.346,-8)
\pspolygon(2.748,-8)(0.5,0)(0,0)(2,-8)
  \psline(0.25,0)(2.394,-8)
  \psline(0.25,0)(2.346,-8)

\psline[linewidth=1.2pt,arrows=<->,arrowscale=3](12,-8)(12,0)
\psline[linewidth=1.2pt,arrows=<->,arrowscale=3](12,0)(12,4)
\psline[linewidth=1.2pt,arrows=<->,arrowscale=3](12,4)(12,20)
\psline[linewidth=2pt,linestyle=dotted,dotsep=8pt](2,0)(12,0)
\psline[linewidth=2pt,linestyle=dotted,dotsep=8pt](6,-8)(12,-8)
\psline[linewidth=2pt,linestyle=dotted,dotsep=8pt](0,4)(12,4)
\psline[linewidth=2pt,linestyle=dotted,dotsep=8pt](-12,20)(12,20)
\rput(-11,19.2){\scalebox{3.3}{$P$}}
\rput(10,19.2){\scalebox{3.3}{$Q$}}
\rput(0.1,5.3){\scalebox{3}{$O$}}
\rput(-2.8,0.6){\scalebox{3}{$M$}}
\rput(2.8,0.7){\scalebox{3}{$N$}}
\rput(-6.6,-6.6){\scalebox{3}{$\tilde M$}}
\rput(6.6,-6.6){\scalebox{3}{$\tilde N$}}

\rput(-10,-1.5){\scalebox{3.3}{(e) $m=4$}}
\rput(13.5,-4.5){\scalebox{3.3}{$\sqrt m$}}
\rput(12.7,2){\scalebox{3.3}{$1$}}
\rput(13,11.5){\scalebox{3.3}{$m$}}
}}

\end{picture}
\caption{Families of the 1st order, with $m = 0,\, 1,\, 2,\, 3,\, 4$. In Fig.~(e) the generating triangle of the family is $MON$, the set $\Convex$ is the triangle $\tilde{M}O\tilde{N}$, and the focal set lies on the line $PQ$.}
\label{fig_set1order}
\end{figure}

Note in passing that all families of the 1st order with fixed $m$ are linearly isomorphic. That is, for any two such families there exists a linear transformation that takes a set (the generating triangle, the $i$th triangle, the $i$th trapezoid, the $i$th focus, $i = 1,\ldots,2^m$) of the former family to the corresponding set of the latter one.

For the ratios of the elementary pairs in the family we have the estimate
\beq\label{for1}
\vk_i \le \frac{\sqrt{m} + 2^{-m}d}{m + 1 + \sqrt{m}} \to 0 \quad \text{as} \ \, m \to \infty.
\eeq
Further, for the areas of the union of triangles and the union of trapezoids we have
\beq\label{for2}
\frac{|\cup_{i=1}^{2^m}\triangle_i|}{|\cup_{i=1}^{2^m} \trap_i|} \le \frac{\ln m + 3/2}{\sqrt{m}}  \to 0 \quad \text{as} \ \, m \to \infty.
\eeq

\begin{remark}\label{zam 7}\rm
The union of trapezoids $\cup_{i=1}^{2^m} \trap_i$ is interpreted as the support of mirrors, and the union of triangles $\cup_{i=1}^{2^m}\triangle_i$ as the support of valleys in the family. Formula (\ref{for2}) means that as $m \to \infty$, the relative area of the support of valleys vanishes.
The ratios $\vk_i$ govern the resistance of mirrors in the family. Roughly speaking, formula (\ref{for1}) indicates that as $m \to \infty$, the resistances of the mirrors can be made arbitrarily close to their smallest value.
\end{remark}

\begin{remark}\label{zam 8}\rm
The first idea that comes to mind is to prove Theorem \ref{t1} directly by putting a large number of small copies of $(1,m)$-sets inside $\Om$. That is, we first put inside $\Om$ as large copy $F_0$ of a $(1,m)$-set as we can, then put several smaller copies $F_1, \ldots, F_p$ of $(1,m)$-sets inside the uncovered part $\Om \setminus F_0$, then put even smaller copies inside $\Om \setminus (F_0 \cup F_1 \cup \ldots \cup F_p)$, etc. The procedure is repeated until a sufficiently large part of $\Om$ is covered. The focal sets of all the copies should lie outside $\Om$. The problem, however, is that the parameter $m$ should go to infinity in this hierarchy of copies. As a result, the sets in the hierarchy become more and more complicated, and one cannot guarantee that the area of the uncovered part of $\Om$ goes to zero.
\end{remark}

We will use instead a more sophisticated construction. Take a $(1,m)$-set and consider the convex hall $\Convex$ of the union of the generating triangle $MON$ and the trapezoids of the family. It is the triangle homothetic to the generating one, with the ratio $\sqrt m + 1$ and with center of homothety at the vertex $O$ of the generating triangle. The height of $\Convex$ equals $\sqrt m + 1$. The focal set of the 1st order lies on the base of another triangle homothetic to $\triangle MON$, with the same center of homothety $O$ and with the ratio $m$ (see Fig.~\ref{fig_set1order}).

Notice that the set of trapezoids is very "sparse" in $\Convex$: the total area of the trapezoids is $\sqrt m d (1+o(1))$,\, $m \to \infty$, whereas the area of $\Convex$ is greater than $md/2$. Actually, the part of $\Convex$ occupied with the $(1,m)$-set is $\Convex$ minus the union of $2^m - 1$ angles with the vertices on the base of the generating triangle. These angles are called {\it angles associated with the $(1,m)$-set}; they do not mutually intersect.

Recall that each set obtained from a $(1,m)$-set by a linear transformation is again a $(1,m)$-set. Apply the following iterative procedure. At the first step take $2^m - 1$ sets obtained from the original $(1,m)$-set by linear transformations that take the generating triangle $MON$ to triangles, with height 1 and with the base parallel to $MN$, inscribed in the $2^m - 1$ angles associated with the original $(1,m)$-set (so that the vertices of the resulting generating triangles coincide with the vertices of the associated angles). The union of all the $2^m$ sets (including the original one) contains the triangle homothetic to the original generating triangle $MON$ with the ratio 2 and the same center $O$. In Fig.~\ref{fig_s2order} there are shown the original $(1,2)$-set and one of the three new $(1,2)$-sets with the generating triangle inscribed into the central associated angle.

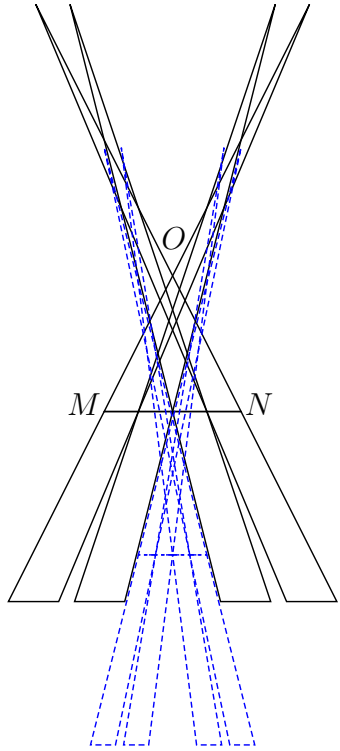
\begin{figure}[h]
\begin{picture}(0,290)

\rput(7,4.4){
\scalebox{0.45}{
\pspolygon[linewidth=1.2pt](-2,0)(4,12)(-1,0)
\pspolygon[linewidth=1.2pt](-1,0)(3,12)(0,0)
\pspolygon[linewidth=1.2pt](2,0)(-4,12)(1,0)
\pspolygon[linewidth=1.2pt](1,0)(-3,12)(0,0)
  \pspolygon[linewidth=1.2pt](-2,0)(-4.8,-5.6)(-3.33,-5.6)(-1,0)
  \pspolygon[linewidth=1.2pt](2,0)(4.8,-5.6)(3.33,-5.6)(1,0)
  \pspolygon[linewidth=1.2pt](-1,0)(-2.866,-5.6)(-1.4,-5.6)(0,0)
  \pspolygon[linewidth=1.2pt](1,0)(2.866,-5.6)(1.4,-5.6)(0,0)
}
\rput(-1.3,0.1){$M$}
\rput(1,0.1){$N$}
\rput(-0.12,2.3){$O$}
}

\rput(7,2.5){
\scalebox{0.45}{
\pspolygon[linecolor=blue,linestyle=dashed,linewidth=1.2pt](-1,0)(2,12)(-0.5,0)
\pspolygon[linecolor=blue,linestyle=dashed,linewidth=1.2pt](-0.5,0)(1.5,12)(0,0)
\pspolygon[linecolor=blue,linestyle=dashed,linewidth=1.2pt](1,0)(-2,12)(0.5,0)
\pspolygon[linecolor=blue,linestyle=dashed,linewidth=1.2pt](0.5,0)(-1.5,12)(0,0)
  \pspolygon[linecolor=blue,linestyle=dashed,linewidth=1.2pt](-1,0)(-2.4,-5.6)(-1.67,-5.6)(-0.5,0)
  \pspolygon[linecolor=blue,linestyle=dashed,linewidth=1.2pt](1,0)(2.4,-5.6)(1.67,-5.6)(0.5,0)
  \pspolygon[linecolor=blue,linestyle=dashed,linewidth=1.2pt](-0.5,0)(-1.433,-5.6)(-0.7,-5.6)(0,0)
  \pspolygon[linecolor=blue,linestyle=dashed,linewidth=1.2pt](0.5,0)(1.433,-5.6)(0.7,-5.6)(0,0)
}
}
\end{picture}
\caption{Starting the construction of a family of the 2nd order.}
\label{fig_s2order}
\end{figure}

At the second step we repeat the procedure as applied to each of the new $(1,m)$-sets. As a result we obtain $(2^m - 1)^2$ copies of the original $(1,m)$-set that fit into the new $(2^m - 1)^2$ associated angles. The union of all the obtained sets (their total number is $1 + (2^m - 1) + (2^m - 1)^2$) contains the triangle homothetic to $\triangle MON$ with the ratio 3.

Note in passing that te height of the generating triangle of each copy obtained this way is 1, and the base is smaller than $d$. Therefore the ratios of all obtained elementary pairs satisfy (\ref{for1}).

We repeat this procedure $\lfloor \sqrt{m} \rfloor + 1$ times, where $\lfloor \ldots \rfloor$ means the integer part. As a result, $\Convex$ will be contained in the union of all the obtained $(1,m)$-sets, including the original one and the ones obtained at the steps $1,\, 2,\ldots,\lfloor \sqrt{m} \rfloor + 1$.

\begin{definition}\label{o_family2step}\rm
The resulting family of elementary pairs is called a {\it family of the 2nd order}. The union of the obtained $(1,m)$-sets is called a $(2,m)$-{\it set}.
The union of the corresponding focal sets is called the {\it focal set of the 2nd order}.
\end{definition}

Note that the image of a family of the 2nd order under a linear transformation is again a family of the 2nd order.


Let us characterize in more detail the linear transformations that take the original $(1,m)$-set to the new ones at each step of the procedure. Each of these transformations can be decomposed into two ones. The first transformation preserves $O$ and transforms the segment $MN$ into itself. The second one is a translation that moves $O$ into the triangle $\Convex$. In other words, it is the translation by a vector $\overrightarrow{OA}$, where $A \in \Convex$.

The focal set of the 2nd order $F$ is the union of the original focal set of the 1st order and its images under these transformations. Let $[M'N']$ be the image of the segment $[MN]$ under the homothety with center at $O$ and the ratio $-m$. The original focal set of the 1st order lies on $[M'N']$. Any other point of $F$ is thus obtained from a point of the segment $[M'N']$ in the following way: first move it to another point of the segment and then translate it by a vector $\overrightarrow{OA}$, with $A \in \Convex$.

Thus, $F$ can be characterized as follows. Put the origin at $O$ and consider the algebraic sum $[M'N'] + \Convex$ of the segment $[M'N']$ and the triangle $\Convex$. Then
$$
F \subset [M'N'] + \Convex.
$$
The domain $[M'N'] + \Convex$ is actually a trapezoid; it is shown bounded by a dashed line in Fig.~\ref{fig_set2order}. It belongs to an $O(\sqrt m)$-neighborhood of the segment $[M'N']$.

By the construction, the $(2,m)$-set contains $\Convex$, and its elementary trapezoids are all contained in the triangle $2\Convex$ (where again the origin is at $O$); see Fig.~\ref{fig_set2order}.

\begin{figure}[h]
\begin{picture}(0,210)

\rput(6.5,2.9){
\scalebox{0.7}{
  \pspolygon[linewidth=0.7pt,fillstyle=solid,fillcolor=lightgray](1.1,-2.2)(-1.1,-2.2)(0,0)
\psline[linewidth=0.7pt](-3,6)(1.1,-2.2)(-1.1,-2.2)(3,6)
  \psline[linewidth=0.7pt](0.3,-0.6)(-0.3,-0.6)
  \rput(-0.66,-0.5){\scalebox{1.4}{$M$}}
  \pspolygon[linewidth=0.7pt](2.2,-4.4)(-2.2,-4.4)(0,0)
\rput(0.66,-0.5){\scalebox{1.4}{$N$}}
\pscircle[fillstyle=solid,fillcolor=gray,linewidth=0.4pt](0,-1.52){0.67}
\pspolygon[linewidth=0.4pt,linestyle=dashed](-3,6)(3,6)(4.1,3.8)(-4.1,3.8)
\psline[linestyle=dotted,linewidth=0.5pt](3,6)(5,6)
\psline[linestyle=dotted,linewidth=0.5pt](0,0)(5,0)
\psline[linestyle=dotted,linewidth=0.5pt](1.1,-2.2)(5,-2.2)
\psline[linewidth=0.3pt,arrows=<->,arrowscale=1.5](5,6)(5,0)
\psline[linewidth=0.3pt,arrows=<->,arrowscale=1.5](5,-2.2)(5,0)
\rput(0,6.3){\scalebox{1.4}{$m$}}
\rput(5.35,2.5){\scalebox{1.4}{$m$}}
\rput(6,-1.1){\scalebox{1.4}{$\sqrt m + 1$}}
\rput(-6.1,5){\scalebox{1.4}{$\sqrt m + 1$}}
\psline[linewidth=0.3pt,arrows=<->,arrowscale=1.5](-5,6)(-5,3.8)
\rput(0,0.6){\scalebox{1.4}{$O$}}
\rput(-3.2,6.2){\scalebox{1.4}{$N'$}}
\rput(3.3,6.3){\scalebox{1.4}{$M'$}}
\rput(-1.5,-1.3){\scalebox{2.2}{$\Convex$}}
\rput(0,-3.5){\scalebox{2.2}{$2\Convex$}}
\rput(0,5.3){\scalebox{1.4}{Focal set}}
\rput(0,4.7){\scalebox{1.4}{of the 2nd order}}
}}

\end{picture}
\caption{The $(2,m)$-set contains the shaded triangle $\Convex$. Its focal set is situated within the dashed line. The elementary trapezoids of the family are all contained in the triangle $2\Convex$.}
\label{fig_set2order}
\end{figure}
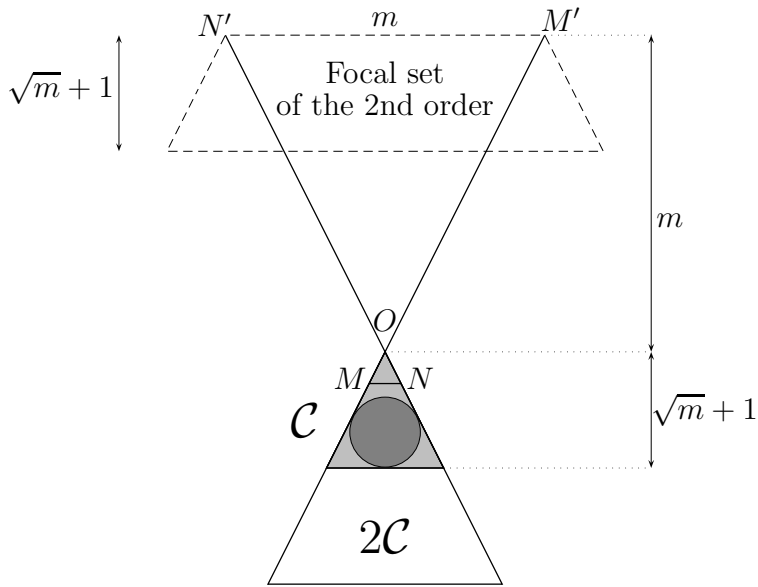

The elementary trapezoids of the family do not mutually intersect and are contained in the triangle $2\Convex$ with height $2\sqrt m + 2$ and base $(2\sqrt m + 2)d$. Therefore the sum of their areas is smaller than $(2\sqrt m + 2)^2 d/2$. The family of elementary trapezoids is divided into several sub-families of the 1nd order. For each sub-family of trapezoids and the corresponding sub-family of elementary triangles, inequality (\ref{for2}) is valid. Summing over all sub-families, we conclude that the area of the union of elementary triangles in the family of the 2nd order is smaller than
\beq\label{eltr}
\frac{\ln m + 3/2}{\sqrt m}\, \frac{(2\sqrt m + 2)^2 d}{2} = 2d\sqrt m\ln m (1 + o(1)), \quad m \to \infty.
\eeq

The circle inscribed in $\Convex$ has the radius $r_1 \sqrt m(1 + o(1))$, and the minimal concentric circle containing $2\Convex$ has the radius $r_2 \sqrt m(1 + o(1)), \ m \to \infty$, where $r_1$ and $r_2$ are positive constants that depend only on the generating triangle.

Now introduce a positive parameter $\om$ and let $m = \lfloor c^2/\om^2 \rfloor$, the constant $c > 0$ to be specified below. Let the generating triangle $MON$ be homothetic, with a negative ratio $-r$, to the triangle $\AAA\BBB\CCC$ indicated in Lemma \ref{l2main}. Then $\triangle M'ON'$ is homothetic to $\triangle \AAA\BBB\CCC$ with the ratio $mr.$ Apply to the family of the 2nd order the composition of two transformations. The first one is the translation that takes the center of the inscribed circle to $O$. (The translation distance is $O(\sqrt m)$.) The second one is a homothety with the ratio $1/(mr)$ that takes $O$ to $\CCC$, $M'$ to $\AAA$, and $N'$ to $\BBB$. The image is another family of the 2nd order (let it be denoted by $\{ \triangle_i^\om,\, \trap_i^\om \}$) that satisfies the following properties.
\vspace{1mm}

(a) The union $\cup_i \overline{(\triangle_i^\om \cup \trap_i^\om)}$ of the obtained elementary sets contains the circle centered at $C$ with radius $\frac{\om}{c} \frac{r_1}{r} (1 + o(1)), \ \om \to 0.$ Take $c < r_1/r$; then for $\om$ sufficiently small the union contains $B_\om(C).$
\vspace{1mm}

(b) Since the ratio of homothety is $1/(mr)$, the area $|\cup_i \triangle_i^\om|$ of the union of triangles in the resulting family is smaller than  $1/(mr)^2$ times the expression in (\ref{eltr}). Taking into account that $1/m \sim \om^2$, we conclude that
$$
\frac{|\cup_i \triangle_i^\om|}{\om^2} \sim \om \ln\frac{1}{\om} \to 0 \quad \text{as} \ \, \om \to 0.
$$

(c) The ratio of an elementary pair is invariant under a homothety; therefore the ratios of all elementary pairs of the family satisfy (\ref{for1}). Thus, we have
$$
\max_i \vk_i^\om \le \frac{\sqrt{m} + 2^{-m}d}{m + 1 + \sqrt{m}} \sim \om \to 0 \quad \text{as} \ \, \om \to 0.
$$

(d) The focal set of the family is contained in the image of the trapezoid $[M'N'] + \Convex$, which in turn belongs to an $O(1/\sqrt m)$-neighborhood of the segment $[\AAA\BBB]$. Taking into account that $1/\sqrt m \sim \om$, we conclude that the focal set belongs to the $O(\om)$-neighborhood of $[\AAA\BBB]$.
\vspace{1mm}

(e) The union $\cup_i \trap_i^\om$ of the trapezoids of the family belongs to the image of $2\Convex$, which in turn belongs to the circle centered at $C$ with radius $\frac{\om}{c} \frac{r_2}{r} (1 + o(1)), \ \om \to 0.$ That is, $\cup_i \trap_i^\om \subset B_{O(\om)}(C).$
\vspace{1mm}

Thus, Lemma \ref{l2main} is proved.

\section{Proofs of Theorems \ref{t2} and \ref{t3}}

\subsection{Theorem \ref{t2}}

The second equality in (\ref{ft2}) is a consequence of Theorem \ref{t1} and Proposition \ref{propo}. It remains to prove the first one.

The trajectory of a billiard particle is naturally parameterized by the time $t$. Let a particle initially move according to $x(t) = x_0, \ z(t) = -t \ (x_0 \in \Om)$, then make several (finitely many) reflections from the graph of $B_u$, and its final velocity be not vertical, $v^+(t) \neq (0,0,1)$. At a point, say $x_1$, the $x$-projection $x(t)$ of the particle leaves $\Om$ (this implies that $x_1 \in \pl\Om$). Reparameterize the part of the trajectory between the point of the first impact and the point where the $x$-projection leaves $\Om$, the parameter being the path length $s$ of the $x$-projection between $x_0$ and the current point.

Consider the $z$-coordinate $z(s)$ of the particle as a function of $s, \ 0 \le s \le s_0$; here $s_0$ is the total length of the $x$-projection (which is a broken line) between $x_0$ and $x_1$; see Fig.~\ref{fig_parameter}. The breaks of the line correspond to the points of reflection of the particle.

\begin{figure}[h]
\begin{picture}(0,155)

\rput(4,0.7){
\scalebox{1}{
\psline[linewidth=0.6pt,arrows=->,arrowscale=1.2](-1,0)(7,0)
\psline[linewidth=0.4pt,linecolor=red,linestyle=dashed,arrows=->,arrowscale=1.5](0,4.5)(0,3)
\psline[linewidth=0.7pt,linestyle=dotted](0,3)(0,0)
\psline[linewidth=0.8pt,linecolor=red,arrows=->,arrowscale=1.5](0,3)(0.5,2)(1.5,1)(3,0.5)
\psline[linewidth=0.4pt,linecolor=red,linestyle=dashed,arrows=->,arrowscale=1.5](3,0.5)(6,-0.5)
\psline[linewidth=0.7pt,linestyle=dotted](0.5,2)(0.5,0)
\psline[linewidth=0.7pt,linestyle=dotted](1.5,1)(1.5,0)
\psline[linewidth=0.7pt,linestyle=dotted](3,0.5)(3,0)
\rput(0,-0.3){\scalebox{1}{$0$}}
\rput(6.7,-0.25){\scalebox{1}{$s$}}
\rput(3,-0.3){\scalebox{1}{$s_0$}}
\rput(1.8,1.9){\scalebox{1}{$z(s)$}}
}}

\end{picture}
\caption{A particle trajectory parameterized by the path length. The breaks of the line correspond to reflections of the particle.}
\label{fig_parameter}
\end{figure}
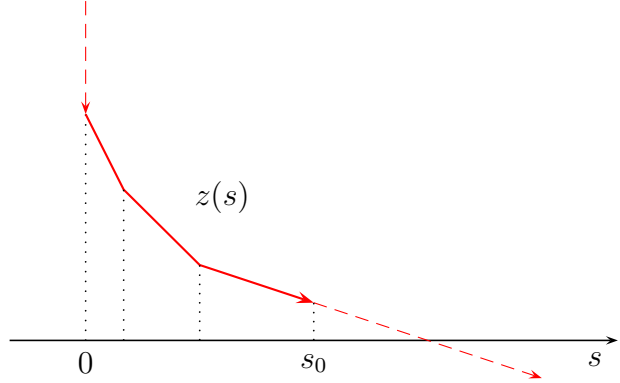

Note that $s_0 \ge \text{dist}(x_0, \pl\Om)$ and $0 \le z(s) \le M$ for all $0 \le s \le s_0$.

Let us show that the $z$-coordinate $v_3$ of the velocity $v$ of the particle does not decrease at each impact. Indeed, for the velocities $v$ and $v^+$ before and after the impact we have
$$
v^+ = v - 2\langle v,\, n \rangle n,
$$
and therefore, $v_3^+ = v_3 - 2\langle v,\, n \rangle n_3,$ where $n$ is the normal to $\pl(B_u)$ at the point of impact and $\langle v,\, n \rangle < 0$. There may be two cases:

(i) the particle reflects from the graph of $u$; then $n$ is the upper normal to graph$(u)$,
$$
n = \frac{(-\nabla u(x), \ 1)}{\sqrt{1 + |\nabla u(x)|^2}};
$$

(ii) the projection of the reflection point is a point of discontinuity of graph$(u)$; that is, the particle is reflected from a "vertical wall". In this case the 3rd component of $n$ is zero, $n = (n_1, n_2, 0)$.

In the case (i), $v_3^+ - v_3 > 0$; that is, the 3rd component of the velocity increases. In the case (i), $v_3^+ - v_3 = 0$; that is, the 3rd component of the velocity remains constant.

The following useful formula relates the derivative $z'(s)$ and the velocity $v$ of the particle at the corresponding point:
$$
\frac{(1, \ z'(s))}{\sqrt{1 + z'^2(s)}} = \Big(\sqrt{v_1^2 + v_2^2}, \ v_3\Big).
$$
It implies that the function $z'(s)$ is constant between impacts, and the increment of $z'$ is nonnegative at each impact. Thus, the function $z(s)$ is convex, and therefore
$$
z'(s_0) \ge \frac{z(s_0) - z(0)}{s_0} \ge -\frac{M}{d(x)}.
$$
Recall the brief notation $d(x) = \text{dist}(x, \pl\Om)$. Therefore the third component of the final velocity $v^+(x,u)$ of the particle satisfies
\beq\label{ineq2}
v_3^+(x,u) = \frac{z'(s_0)}{\sqrt{1 + z'^2(s_0)}} \ge -\frac{M/d(x)}{\sqrt{1 + M^2/d^2(x)}} = -\frac{M}{\sqrt{M^2 + d^2(x)}}.
\eeq
If, on the other hand, the final velocity is vertical then its 3rd component $v_3^+ = 1$ obviously satisfies (\ref{ineq2}).

Combining (\ref{o Res2}), (\ref{o_phi}), and (\ref{ineq2}), we obtain
$$
F(u) \ge \phi(\Om,M) \quad \text{for all} \quad u \in \UUU_{\Om,M}.
$$
On the other hand, we have $\SSS_{\Om,M} \subset \UUU_{\Om,M}$, therefore by Theorem \ref{t1}
$$
\inf_{u \in \UUU_{\Om,M}} F(u) \le \inf_{u \in \SSS_{\Om,M}} F(u) = \phi(\Om,M).
$$
Thereby the first relation in (\ref{ft2}) is also proved.

\subsection{Theorem \ref{t3}}

Since $D_{\Om,M} \subset \BBB_{\Om,M},$ it suffices to prove the second equality in (\ref{ft3}).

Let $\NNN_\ve(A)$ denote the $\ve$-neighborhood of the set $A$. Fix $\ve > 0$ and denote
$$
\tilde\Om = \Om \setminus \NNN_\ve(\pl\Om).
$$
By Lemma \ref{l_4} there exists a finite set of points $O_i \in \Om \cap \NNN_\ve(\pl\Om)$, a domain $V \subset \tilde\Om$ with $|V| < \ve$, and a function $u \in \SSS_{\tilde\Om,M/2}$ such that the particle corresponding to a point $x \in V$ is reflected vertically from graph$(u)$, and the particle corresponding to a regular point of $\tilde\Om \setminus V$ after the reflection passes through a point $(O_i, 0)$.

For each $i$ find a ball $\UUU_i \subset \RRR^2 \setminus \Om$ such that the distance between $O_i$ and each point of $\UUU_i$ is smaller than $\ve$ (see Fig.~\ref{fig_Omega}). In other words, $\UUU_i \subset B_\ve(O_i) \setminus \Om.$
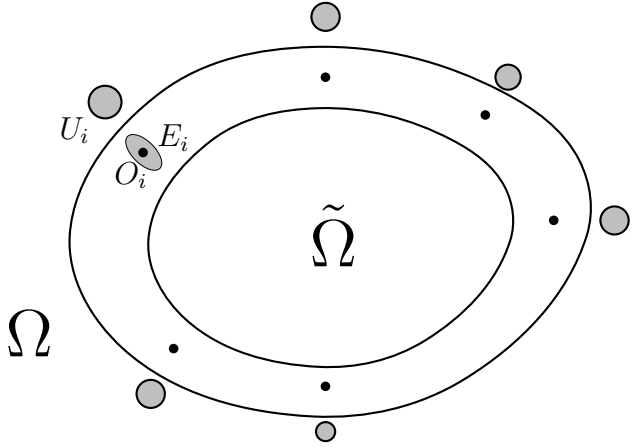
\begin{figure}[h]
\begin{picture}(0,175)

\rput(6,0.5){
\scalebox{1}{
\rput(0,-0.2){
\psecurve(-2.5,1)(0,0.1)(2,0.6)(3.6,2.5)(3.5,3.5)(2.5,4.4)(0,5)(-2,4.4)(-3.2,2.5)(-2.5,1)(0,0.1)(2,0.6)}
\scalebox{0.7}{
\rput(0,0.7){
\psecurve[linewidth=1.14pt](-2.5,1)(0,0.1)(2,0.6)(3.6,2.5)(3.5,3.5)(2.5,4.4)(0,5)(-2,4.4)(-3.2,2.5)(-2.5,1)(0,0.1)(2,0.6)
}}
\rput(0,-0.2){\rput{-45}(-3.1,-0.5){\psellipse[linewidth=0pt,fillstyle=solid,fillcolor=lightgray](-2.4,3.4)(0.3,0.15)}}
\psdots(-2,0.8)(0,0.3)(3,2.5)(2.1,3.9)(0,4.4)(-2.4,3.4)
\pscircle[fillstyle=solid,fillcolor=lightgray](-2.3,0.2){0.2}
\pscircle[fillstyle=solid,fillcolor=lightgray](0,-0.3){0.14}
\pscircle[fillstyle=solid,fillcolor=lightgray](3.8,2.5){0.2}
\pscircle[fillstyle=solid,fillcolor=lightgray](2.4,4.4){0.18}
\pscircle[fillstyle=solid,fillcolor=lightgray](0,5.2){0.2}
\pscircle[fillstyle=solid,fillcolor=lightgray](-2.9,4.07){0.23}
\rput(-3.9,1){\scalebox{2.1}{$\Om$}}
\rput(0.1,2.3){\scalebox{2.1}{$\tilde\Om$}}
\rput(-2.56,3.1){$O_i$}
\rput(-2,3.6){$E_i$}
\rput(-3.3,3.7){$\UUU_i$}
}}

\end{picture}
\caption{A part of the construction on the $x$-plane.}
\label{fig_Omega}
\end{figure}
Let $K_i$ be the cone with vertex at $(O_i, M/2)$ and with the base $\UUU_i \times \{ 0 \}$. That is, $K_i$ is the union of rays with the endpoints at $(O_i, M/2)$ through all the points $(x,0), \ x \in \UUU_i.$ Let $D_i$ be the intersection of $K_i$ with $\pl\Om \times [0,\, M/2].$

Then for each $i$ we choose $x_i \in \UUU_i$ and denote by $l_i$ the ray with the endpoint $(O_i, M/2)$ through $(x_i,0).$ We select the points $x_i$ in such a way that the resulting rays $l_i$ do not mutually intersect.

Next for each $i$ we choose a paraboloid of rotation $P_i$ with focus at $(O_i, M/2)$, with the axis containing $l_i$, and such that $l_i \cap P_i = \emptyset.$ Denote by $\PPP_i$ the convex hull of $P_i$; that is, $\PPP_i$ is the convex closed domain bounded by the paraboloid $P_i.$ Our choice implies that the ray $l_i$ is contained in the axis of $P_i$ and in the domain $\PPP_i$. Denote $E_i = \PPP_i \cap (\RRR^2 \times \{ M/2 \})$. Thus, $E_i$ is an ellipse (with its interior) containing $O_i$ on the horizontal plane $z = M/2$ (see Fig.~\ref{fig_Omega}). We impose the additional conditions:

(i) the domains $\PPP_i$ are mutually disjoint;

(ii) $\PPP_i \cap (\pl\Om \times \RRR) \subset D_i;$

(iii) $E_i \subset \Om \setminus \overline{\tilde{\Om}};$ that is, the ellipses $E_i$ do not intersect $\pl\Om$ and $\pl(\tilde\Om).$

The conditions (i)--(iii) mean that the paraboloids $P_i$ should be "sufficiently thin".

Take a body $B_\ve$ which is the union of four domains $B_i = B_i^\ve, \ 1 \le i \le 4,$
$$
B_\ve = B_1 \cup B_2 \cup B_3 \cup B_4
$$
(see Fig~\ref{fig_final}). Here $B_1$ is the subgraph of the function $M/2 + u$,
$$
B_1 = \{ (x,z) : \, x \in \overline{\tilde\Om}, \ 0 \le z \le M/2 + u(x) \},
$$
and
$$
B_2 = (\bar{\Om} \times [0,\, M/2]) \setminus ( \cup_i \PPP_i ).
$$
Further, we take two open domains $\Om_0 \subset \Om_1 \subset \Om \setminus \tilde{\Om}$ such that $\pl\Om_1 \cap \pl\tilde\Om = \emptyset$ and for each $i, \ \Om_0 \cap E_i = \emptyset$ and $\cup_i E_i \subset \Om_1.$ One can take, for instance, $\Om_0 = \Om \cap \NNN_{\ve'}(\pl\Om)$ and $\Om_1 = \Om \setminus \overline{\NNN_{\ve'}(\tilde\Om)}$ with $\ve' > 0$ sufficiently small.

We define
$$
B_3 = \bar{\Om}_0 \times [0,\, M],
$$
$$
B_4 = \bar{\Om}_1 \times [M - \ve',\, M],
$$
where $\ve'$ is taken so small that the domain $B_4$ does not intersect Conv$\big( (\tilde\Om \times [0,\, M]) \cup (\Om \times [0,\, M/2]) \big),$ and therefore is inaccessible for trajectories of the particles reflected from graph$(M/2 + u).$ The domain $B_4$ serves to "shield" the planar domains $E_i$ (which are inlets of the "hollows" $\PPP_i \cap \{ z \le M/2 \}$) from incident particles with vertical direction. The domain $B_3$ serves to make the whole body $B_\ve$ connected.

We have $\bar{\Om} \times \{0\} \subset B_\ve \subset \bar{\Om} \times [0,\, M]$. The surface $\pl B_\ve \setminus (\pl\Om \times \RRR)$ is piecewise smooth by the definition of $B_\ve$. As we will see below, the billiard scattering outside $B_\ve$ is regular. Therefore $B_\ve \in \BBB_{\Om,M}.$

The domains $B_3 \cup B_4$ and $B_1$ are obviously connected. The section of $B_2$ by a horizontal plane $z = c, \ 0 \le c \le M/2$ is also connected and has nonempty intersection with $B_1$; therefore $B_1 \cup B_2$ is connected. Since the domains $B_3 \cup B_4$ and $B_1 \cup B_2$ have nonempty intersection, their union $B_\ve$ is connected.

If a flow particle incident on $B_\ve$ corresponds to a regular point of $\Om \setminus \overline{\tilde{\Om}},$ then either $x \in \Om_1$, or $x \in (\Om \setminus \overline{\tilde{\Om}}) \setminus \Om_1.$ In both cases the particle is reflected vertically, and so, $v_3(x; B_\ve) = 1.$

If a particle corresponds to a regular point of $\tilde\Om$, then either $x \in V$, or $x$ is a regular point of $\tilde\Om \setminus V.$ In the former case the particle is reflected vertically, and in the latter case the reflected particle passes through a point $(O_i, M/2)$ and then moves in $\PPP_i.$ It may further happen that it makes one more reflection (which is necessarily from $P_i$) and then moves freely parallel to $l_i$ (see Fig.~\ref{fig_final}).

\begin{figure}[h]
\begin{picture}(0,210)
\rput(1.3,0.7){
\pspolygon[fillstyle=solid,fillcolor=gray](0,0)(11,0)(11,3)(0,3)
\pspolygon[fillstyle=solid,fillcolor=gray](0,3)(0,6)(1.5,6)(1.5,5.5)(0.5,5.5)(0.5,3)
\pspolygon[fillstyle=solid,fillcolor=gray](11,3)(11,6)(9.5,6)(9.5,5.5)(10.5,5.5)(10.5,3)
\pscustom[fillstyle=solid,fillcolor=white]{
\pscurve(0,2)(0.4,2.6)(0.75,3)
\pscurve[liftpen=1](1.1,3)(0.6,2)(0,1.1)
}
\pspolygon[linewidth=0pt,linecolor=white,fillstyle=solid,fillcolor=lightgray](2,3)(2,6)(4.5,6)(4.5,3)
\rput(2.25,3.2){$A_1$}
\rput(2.25,5.8){$A_4$}
\rput(4.25,3.2){$A_2$}
\rput(4.25,5.8){$A_3$}

\psline[linewidth=1.2pt,linecolor=white](1.1,3)(0.75,3)
\psline[linewidth=1.2pt,linecolor=white](0,1.2)(0,2)
\psline[linecolor=red,arrows=->,arrowscale=1.5](4.8,7)(4.8,6)
\psline[linecolor=red,arrows=->,arrowscale=1.5](4.8,6)(4.8,5.6)(0.6,2.8)(-0.7,0)
\psdots(0.9,3)
\rput(0.8,3.3){$O_i$}
\psline(-0.5,0)(-0.9,0)
\psline(2,0)(2,3)
\psline(9,0)(9,3)
\rput(-1.2,-0.2){$\UUU_i$}
\rput(-0.45,1.9){$P_i$}
\psline(2,3)(2,6)
\psline(9,3)(9,6)
\psline(4.5,5)(5.05,6)
\psline(5.05,5.15)(5.5,6)
\psline(5.5,6)(5.95,5.15)
\psline(5.95,6)(6.5,5)
\psline(6.5,6)(6.95,5)
\psline(6.95,6)(7.35,5)
\psline(7.35,6)(7.7,5)
\psline(7.7,6)(8,5)
\psline(8,6)(8.25,5)
\psline(8.25,6)(8.5,5)
\psline(8.5,6)(8.7,5)
\psline(8.7,6)(8.85,5)
\psline(8.85,6)(9,5)

\pspolygon[linewidth=0pt,fillstyle=solid,fillcolor=gray](4.5,3)(4.5,5)(5.05,6)(5.05,5.15)(5.5,6)(5.95,5.15)(5.95,6)(6.5,5)
(6.5,6)(6.95,5)
(6.95,6)(7.35,5)
(7.35,6)(7.7,5)
(7.7,6)(8,5)
(8,6)(8.25,5)
(8.25,6)(8.5,5)
(8.5,6)(8.7,5)
(8.7,6)(8.85,5)
(8.85,6)(9,5)(9,3)
\rput(-0.6,4.5){\scalebox{1.5}{$B_3$}}
\psline[linewidth=0.4pt]{<-}(11.6,3)(11.6,4.2)
\psline[linewidth=0.4pt]{->}(11.6,4.8)(11.6,6)
\rput(11.7,4.5){$M/2$}
\psline[linewidth=0.4pt]{<-}(11.6,0)(11.6,1.2)
\psline[linewidth=0.4pt]{->}(11.6,1.8)(11.6,3)
\rput(11.7,1.5){$M/2$}
\psline[linestyle=dotted](11,3)(11.7,3)
\rput(0.75,6.4){\scalebox{1.5}{$B_4$}}
\rput(10.25,6.4){\scalebox{1.5}{$B_4$}}
\rput(6.5,4){\scalebox{1.5}{$B_1$}}
\rput(5.5,1.5){\scalebox{1.5}{$B_2$}}
\psline[linewidth=0.4pt]{<-}(0,-0.5)(0.8,-0.5)
\psline[linewidth=0.4pt]{->}(1.2,-0.5)(2,-0.5)
\rput(1,-0.5){$\ve$}
\psline[linewidth=0.4pt]{<-}(9,-0.5)(9.8,-0.5)
\psline[linewidth=0.4pt]{->}(10.2,-0.5)(11,-0.5)
\rput(10,-0.5){$\ve$}
}
\end{picture}
\caption{The cross section of the body $B_\ve$ by a vertical plane and the trajectory of a particle are shown. An unlikely case when all segments of the trajectory lie in one plane is depicted. The rectangle $A_1A_2A_3A_4$ shown lightgray projects on the valley $V$ and therefore does not contain points of $B_\ve$.}
\label{fig_final}
\end{figure}
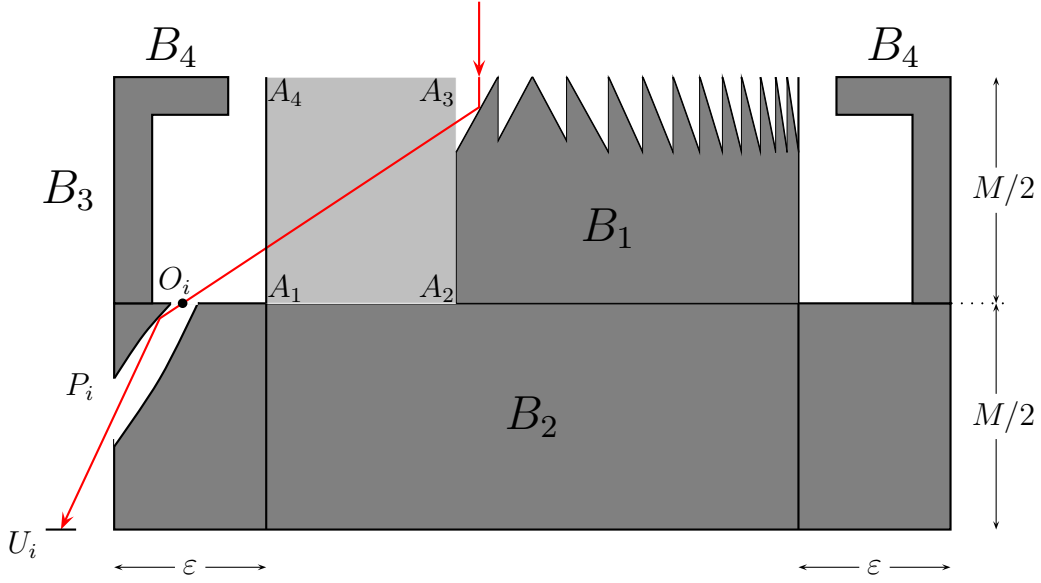

If the particle makes no reflections anymore, then it necessarily intersects $\pl\Om \times [0,\, M/2]$ at a point of $D_i$, and then intersects the disc $\UUU_i \times \{ 0 \}.$ In both cases the final motion is parallel to a line through $(O_i, M/2)$ and $(x,0), \ x \in \UUU_i.$ This implies that the 3rd component of the final velocity $v^+(x; B_\ve)$ satisfies
$$
v_3^+(x; B_\ve) < -\frac{M/2}{\sqrt{M^2/4 + \ve^2}}.
$$

Thus, each particle makes no more than two reflections, and so, $B_\ve$ satisfies DIC. Since $B_\ve$ is connected, we conclude that $B_\ve \in \DDD_{\Om,M}.$

The resistance of $B_\ve$ equals
$$
R(B_\ve) = \int_\Om \frac{1 + v_3^+(x;B_\ve)}{2}\, dx = \int_{(\Om\setminus\tilde\Om)\cup V} (\cdots) + \int_{\tilde\Om\setminus V} (\cdots)
$$
$$
< |\Om\setminus\tilde\Om| + |V| + |\tilde\Om\setminus V| \cdot \frac{1}{2} \left( 1 - \frac{M/2}{\sqrt{M^2/4 + \ve^2}} \right).
$$
Taking into account that $|\Om\setminus\tilde\Om| < \ve |\pl\Om|, \ |V| < \ve,$ and $|\tilde\Om\setminus V| < |\Om|,$ we conclude that $R(B_\ve) \to 0$ as $\ve \to 0.$ Theorem \ref{t3} is proved.

\subsection*{Acknowledgements}

This work was supported by Portuguese funds through CIDMA -- Center for Research and Development in Mathematics and Applications and FCT -- Portuguese Foundation for Science and Technology, within the project PEst-OE/MAT/UI4106/2014, as well as by the FCT research project PTDC/MAT/113470/2009. The author is grateful to
Evgeny Lakshtanov for the help in preparing Figure~\ref{fig4min}.

\end{document}